\documentclass{article}

\usepackage[utf8]{inputenc}
\usepackage{color}
\usepackage{subfigure}
\usepackage{url}
\usepackage{graphicx}
\usepackage{amsmath}
\usepackage{amsthm}
\usepackage{amssymb}
\usepackage{bm}
\usepackage[plainpages=false]{hyperref}
\usepackage{enumerate}

\newcommand*{\doi}[1]{doi: \href{https://dx.doi.org/#1}{\urlstyle{rm}\nolinkurl{#1}}}
\newcommand*{\arxiv}[1]{arXiv:  \href{https://arxiv.org/abs/#1}{\urlstyle{rm}\nolinkurl{#1}}}

\usepackage{titlesec}
\titleformat{\section}
{\normalfont\Large\bfseries\filcenter}{\thesection.}{1 ex}{}
\titleformat{\subsection}
{\normalfont\normalsize\bfseries}{\thesubsection.}{1 ex}{}
\titleformat{\subsubsection}[runin]
{\normalfont\normalsize\bfseries\filcenter}{\thesubsection.}{1 ex}{}

\usepackage[margin=1.2 in]{geometry}

\usepackage[square,numbers,sort&compress]{natbib}

\usepackage{thmtools,thm-restate}
\declaretheorem[within=section]{theorem}

\declaretheorem[sibling=theorem,name=Proposition]{prop}

\declaretheorem[style=remark,sibling=theorem,qed={$\diamondsuit$}]{remark}
\declaretheorem[style=definition]{definition}

\declaretheorem[style=remark,sibling=theorem]{example}

\newcommand{\iprod}[2]{\left\langle {#1},{#2}\right\rangle}

\newcommand\RR{\mathbb{R}}

\newcommand{\Euc}{{\rm Euc}}

\newcommand{\Id}{\mathrm{Id}}

\DeclareMathOperator{\Sym}{Sym}
\DeclareMathOperator{\Aut}{Aut}

\newcommand\p{{\bf p}}
\newcommand\x{{\bf x}}

\newcommand\e{{\bf e}}

\newcommand\q{{\bf q}}

\newcommand{\eps}{\varepsilon}

\usepackage{tikz}
\usetikzlibrary{arrows}
\usepackage[utf8]{inputenc}
\usepackage{tabularx}
\usepackage{algorithm}
\usepackage{algpseudocode}
\usepackage{caption}
\usepackage{graphicx}

\title{Equilibrium stresses in frameworks via symmetric averaging}
\author{Cameron Millar\thanks{Skidmore, Owings \& Merrill,  London,  UK, \texttt{cameron.millar@som.com}}, Bernd Schulze\thanks{School of Mathematical Sciences, Lancaster University,  UK, \texttt{b.schulze@lancaster.ac.uk}}\, and Louis Theran\thanks{School of Mathematics and Statistics, University of St Andrews,  UK, \texttt{louis.theran@st-andrews.ac.uk}}}

\begin{document}

\date{}
\maketitle

\begin{abstract}
    For a bar-joint framework $(G,\p)$, a  subgroup $\Gamma$ of the automorphism group of $G$, and a subgroup of the orthogonal group isomorphic to $\Gamma$, we introduce a symmetric averaging map which produces a  bar-joint framework on $G$ with that symmetry. If the original configuration is ``almost symmetric", then the averaged one will be near  the original configuration.  With a view on structural engineering applications, we then introduce a hierarchy  of definitions of ``localised" and ``non-localised" or ``extensive" self-stresses of frameworks and investigate their behaviour under the symmetric averaging procedure. Finally, we present algorithms for finding non-degenerate symmetric frameworks with many states of self-stress, as well as non-symmetric and  symmetric frameworks with extensive self-stresses.
   The latter uses the symmetric averaging map in combination with  symmetric Maxwell-type character counts and a procedure based on the pure condition from algebraic geometry.
        These algorithms 
    provide new theoretical and computational tools  for the design of engineering structures such as gridshell roofs. 
\end{abstract}

\medskip


\section{Introduction}
A bar-joint framework is, loosely speaking, a collection of rigid bars that are joined at their ends by freely-rotating joints. The  mathematical field of Geometric Rigidity Theory is concerned with the rigidity and flexibility analysis of bar-joint frameworks \cite{CGbook, bernd2017, WW}. This theory has many applications in science, technology and design, since bar-joint frameworks are useful models of a variety of real-world structures. See \cite{mbmm,marina,mmsb,smmb,schmil}, for example, for some recent applications in structural engineering.

Each  bar-joint framework $(G,\p)$ in $\RR^d$ with $n$ vertices and $m$ 
edges has an $s$-dimensional space $S(\p)$ of self-stresses (also known as equilbrium stresses) and an $f$-dimensional 
space of non-trivial infinitesimal motions.  These are related by the Maxwell 
index theorem
\begin{equation}\label{eq:maxwell}
   k=f-s= dn - m - \frac{d(d+1)}{2}.
\end{equation}
Once $\p$ is fixed, all of these numbers are easy to compute with linear 
algebra.  Moreover, it is well-known that, for each fixed $G$ and $d$, 
there is a typical value for $s$, which is achieved at almost every $\p$ (if $k=0$ then typically $f=s=0$, if $k>0$ then typically $f=k, s=0$, and if $k<0$ then typically $f=0, s=-k$).

Recently, a number of different applications, including the theory of 
generic global rigidity, graph realisation, statistics, and gridshell 
design, have suggested a different line of investigation.  Here, 
we allow $\p$ to vary, and instead, ask about the collection of 
vectors $\mathbf \omega\in \RR^m$ that are a self-stress for 
\emph{some} $(G,\p)$.  Research along these lines has antecedents in 
Connelly's work on universal rigidity \cite{connelly82,gt} in the early 1980's, and 
has been active over the past decade or so \cite{Abdo,CGT,GHT,gstr}.

This note is motivated by a specific type of question, which is 

\medskip 
\noindent    {\it For a fixed graph $G$ and dimension $d$, what is the maximum number of independent self-stresses 
    a framework $(G,\p)$ can have?} 
\medskip

Our specific motivation is gridshell design \cite{gridshell,mmsb,smmb}.  Here we imagine a 
framework  as the $1$-skeleton of a $3$-dimensional 
polyhedral surface that models a  roof made of flat glass panels supported by a steel skeleton (a gridshell).
When we consider vertical loads applied at the joints of the skeleton, axial resolutions 
in the bars correspond  to self-stresses in the $2$-dimensional planar
framework (also known as the form diagram) obtained by orthogonally projecting the 3D structure to the $xy$-plane.  (The actual shape of the polyhedral surface is also related to 
the self-stresses in the projection via Maxwell--Cremona liftings \cite{maxwell1864xlv, maxwell1870reciprocal,cremona1872figure, WW82}.) 
A central design constraint for a gridshell is that 
it axially resolves a certain set of dominant 
vertical loads.  To ensure this, designers work 
in reverse, starting from the $2$-dimensional form diagram, and 
build it to have  enough independent self-stresses.
(In general, the gridshell will \emph{not} be able 
to axially resolve \emph{every} load, and possibly not 
even every vertical load.  Modern gridshells have too 
few members to be statically rigid.  This is a source of 
subtlety in the problem.)

Figure \ref{fig:Gridshell} shows a hypothetical gridshell on a rounded rectangle boundary with dihedral symmetry of order 4.
The gridshell has been designed so that it has axial forces only (is funicular) under its self-weight. It can also be thought of as a hanging cable net which has been inverted - ``As hangs a flexible cable so, inverted, stand the touching pieces of an arch". Many of these structures are considered to be undercounted. Gridshells are often quad-dominant as quadrilateral glass panels are more economical (reduced offcut wastage) and surfaces with planar quad discretisations are eminently realisable. Due to their undercounting, efficiently carrying unbalanced loads, such as wind and snow drifts, can be challenging. By considering states of self-stress in the form diagram (each state of self-self corresponds directly to a funicular load case for the gridshell) and their symmetry types, it is possible to design a gridshell to use material efficiently. This paper builds on previous efforts from rigidity theory to provide engineers with tools to design in a mathematically informed manner. 

\begin{figure}[h]
    \centering
    \includegraphics[width=0.7\textwidth]{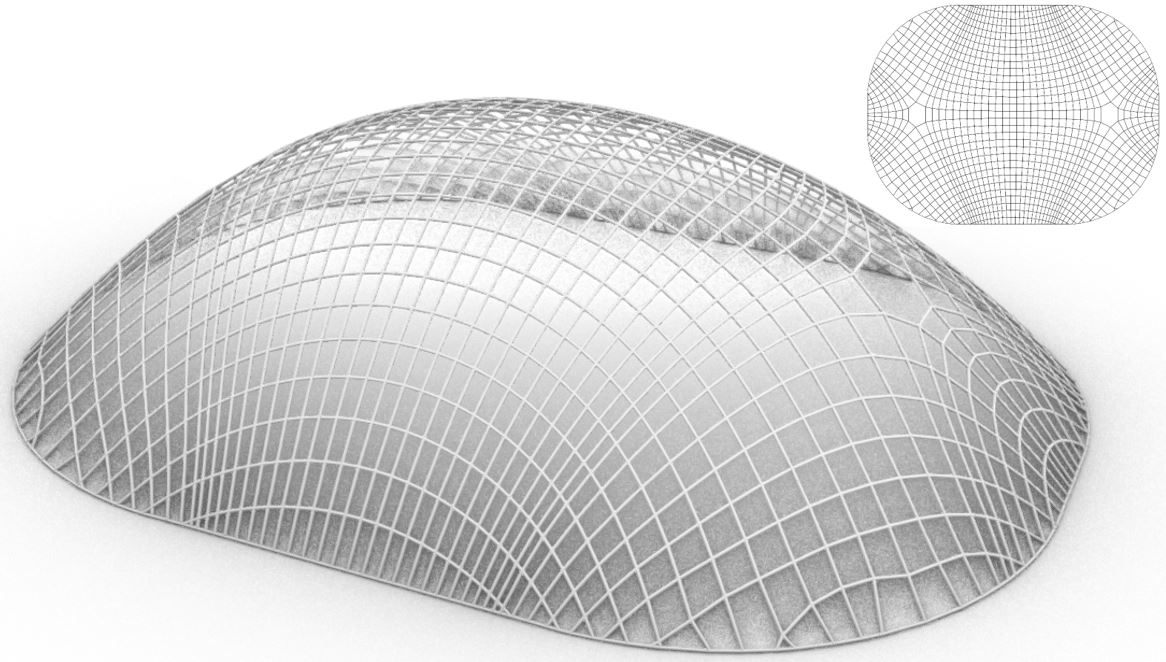}
    \caption{\centering Hypothetical gridshell design with form diagram (plan/top view) inset. \newline Note that this gridshell is not self-tied and relies upon pin supports along the boundary. }
    \label{fig:Gridshell}
\end{figure}

It is not hard to see that, in order to make the above question about maximising the number of self-stresses substantial, 
we need some restriction on $\p$.  Otherwise, we just make $\p_i = \p_j$
for all points $i$ and $j$ in the configuration $\p$.  Then $s = m$, which 
by the Maxwell index theorem is the largest possible.
One direction, which we will not pursuse, is to try and
enumerate a large set of application-specific desiderata.
These might include, for example, 
planar embedding, convex faces, no small 
angles, and so on.  Conditions like these are 
semi-algebraic, i.e., involving polynomial inequalities,
which are not present in rigidity-theoretic formalisms.
They are also typically difficult to analyse mathematically, and computationally harder than 
algebraic considerations.

Instead, we will focus on a specific issue, namely how symmetry 
affects the parameter $s$.  Symmetry arises very naturally for 
problems like this in several ways: symmetric designs are aesthetically 
pleasing and have appeared in the architectural pattern language for 
thousands of years; there is also a robust mathematical and engineering 
toolbox for designing and analysing symmetric structures and their 
responses to different real-world conditions. Decomposing loads and self-stresses into different symmetry types can be useful when designing gridshells and other engineering structures (see e.g. \cite{mmzk,schmil,pandia}). Moreover, symmetry often leads to additional states of self-stress that are not present in non-symmetric realisations (see e.g. \cite{fow00,smmb,mmsb}). Our interest here 
is to understand, given some (typically non-symmetric) framework $(G,\p)$ (which we 
imagine to be provided by either an optimisation procedure or a
designer), what the properties of (typically nearby) symmetric frameworks are.

A given graph $G$ has an automorphism group of combinatorial symmetries.
If we fix a specific subgroup isomorphic to a point group, and then associate 
it to the corresponding geometric symmetries of the plane, we 
have fixed a class of symmetric placements of $G$.  Inside of this smaller space, 
we may then seek to maximise $s$.  Because there are fewer total degrees of 
freedom, the symmetry approach can be more computationally tractable.  
Because self-stresses are preserved under projective transformations, 
it is also possible to use symmetry methods to construct self-stressed form diagrams
that are not themselves symmetric.

In this article we will first review the definition of a symmetric framework, pointing out some subtleties (Section~\ref{sec:sym}), and some basic terminology and results from Geometric Rigidity Theory (Section~\ref{sec:rig}). In Section~\ref{sec:symav} we then describe how to find a  symmetric framework to any framework $(G,\p)$ with non-trivial combinatorial symmetries via a novel symmetric averaging procedure. This procedure is inspired by the averaging method popularised by Whiteley in the last century \cite{bernd2017}. If the original framework $(G,\p)$ is ``almost"  symmetric, then the symmetrically averaged framework will be ``nearby"  $(G,\p)$. However, this procedure can also be applied to frameworks that are far from being symmetric, but whose graphs have non-trivial combinatorial symmetries. In Section~\ref{sec:symavprop} we show that the symmetric averaging function actually has the nice property that it is the orthogonal projection onto the space of symmetric configurations; hence it provides the closest symmetric framework to any given framework.
 We then introduce a hierarchy of notions of ``localised" and ``extensive" self-stresses (with respect to a given combinatorial symmetry) and illustrate, with examples and  theorems, the kinds of behaviours that symmetric averaging can  induce (Section~\ref{sec:stresses}). Moreover, we introduce a symmetry-independent notion of an ``extensive" self-stress. Loosely speaking, (symmetry-)extensive self-stresses are the ones that are most relevant to gridshell design, as they have larger support and do not arise from linear combinations of localised self-stresses. 
Section~\ref{sec:alg} is dedicated to the description of three heuristic algorithms, which for a given graph with combinatorial symmetries, allow us to: 
    \begin{itemize} \item [(1)] systematically explore, for a given graph $G$, the symmetries that provide the largest number of independent self-stresses in symmetric  realisations of $G$ in $2$- or $3$-space under symmetry-generic conditions. This relies on the symmetry-adapted Maxwell counts given in \cite{fow00,smmb,mmsb} and also provides information about the symmetry types of the self-stresses; 
    \item [(2)] for an isostatic (i.e. rigid and self-stress-free) graph $G$, check for extensive self-stresses in (not necessarily symmetric) realisations of $G$ in $d$-space via a novel approach based on the pure condition  developed by White and Whiteley \cite{wwstresses}; for $d=2$, we also describe an alternative test based on the force-density method \cite{schek}, which is also known as ``rubber-banding" in Geometric Rigidity Theory \cite{tutte,LLW,lovasz};
    \item [(3)] for an isostatic graph $G$, check for the existence of symmetric realisations  in $\mathbb{R}^2$ that have an extensive  self-stress. Such realisations can then be constructed with the symmetric averaging method. This heuristic combines the ``symmetry-adapted Maxwell counting method" with the ``pure condition method" mentioned above. For simplicity, we focus on the case where the underlying graph is  isostatic. Extensions to more general graphs are left for future work.
   \end{itemize}

  Finally, in Section~\ref{sec:model}  we point out  some subtelties in the relationship between self-stresses in planar form diagrams and the ability of the corresponding $3$-dimensional roof structures to resolve vertical loads. Moreover, in Section~\ref{sec:practice} we analyse how errors in the configurations propagate to resolving vertical loads. These discussions are important for practical applications.

\section{Symmetric graphs and  frameworks}\label{sec:sym}
Let $G = (V,E)$ be a finite simple graph. We will assume that $|V| = n$  and 
$|E| = m$, i.e., that $G$ has $n$ vertices and $m$ edges.
A \emph{framework} in dimension $d$, $(G,\p)$, is a graph $G$ and a 
configuration $\p$ of $n$ points in $\RR^d$.  $(G,\p)$ is also called a \emph{realisation} of the graph $G$ in $\mathbb{R}^d$. (We will come back 
to the indexing of $V$ and $\p$ shortly.)

In this paper, it is convenient to assume that  
$V = \{1, \ldots, n\}$ so that the symmetric group 
$\Sym(n)$ acts on $V$ in the usual way by relabeling the vertices.  
Most vertex relabelings are not very interesting, because they 
change $G$.  A relabeling $\gamma \in \Sym(n)$ is called a 
\emph{graph automorphism} if it preserves $G$ in the following sense:
\[
    ij\in E \qquad \Longleftrightarrow \qquad \gamma(i)\gamma(j)\in E
\]
If $\gamma$ and $\gamma'$ are both graph automorphisms, then so is the 
composition $\gamma\circ \gamma'$ (which we usually write $\gamma\gamma'$).
The set of all the automorphims of $G$ is called $\Aut(G)$, and with 
function composition as the product it becomes a subgroup  of $\Sym(n)$.
So $\Aut(G) < \Sym(n)$ is the \emph{(combinatorial) symmetry group} of $G$.  Note that since $G$ 
doesn't have any geometry, $\Aut(G)$ can be quite large, but it won't be  
$\Sym(n)$, unless $G$ is complete or has no edges.  
Before leaving combinatorial symmetries, we need to look at how 
$\gamma\in \Aut(G)$ moves the edges of $G$.  Since $\gamma$ 
sends edges to edges, it induces a map $\hat{\gamma} : E\to E$, by 
the defining property 
\[
    \hat{\gamma}(ij) = \gamma(i)\gamma(j)\qquad \text{for all $ij\in E$}
\]
So there is a naturally associated subgroup of $\Sym(E)$ that has 
the same structure as $\Aut(G)$, namely 
\[
    \{\hat{\gamma} : \gamma\in \Aut(G)\} < \Sym(E)
\]
We usually don't write $\hat{\gamma}$ and follow the convention that 
$\gamma(ij) = \hat{\gamma}(ij)$.  (We note that the choice of $\hat{\gamma}$
is not the only one we could make, but we will always use this recipe.)

The next step towards symmetric frameworks is to define symmetric point 
configurations.  Let $\p$ be a configuration of $n$ distinct points in $\mathbb{R}^d$.  
In this note, configurations are ordered, so we have 
\[
    \p = (\p_1, \p_2, \ldots, \p_n)\in C_{d,n}
\]
where $C_{d,n}$ is the \emph{configuration space} of all $d$-dimensional 
configurations of $n$ points.  (We give it a name different from $\RR^{dn}$ 
because we ruled out coincident points.)

Next we fix some finite group 
\[
    T = \{\tau_1, \ldots, \tau_g\} < \Euc_0(d)
\]
of (necessarily linear) isometries of $\RR^d$ that fix the origin.  
(These are also known as \emph{point groups} and are classified  in dimensions $2$ and $3$, see e.g. \cite{atk70,alt94,conway}.)  We say that $\p$ is $T$-symmetric if every element of $T$ 
sends points in $\p$ to other points in $\p$; in symbols 
\[
   \text{for all $\p_i$ and $\tau_j\in T$ there exists $\p_k$ such that } \tau_j\p_i = \p_k.\qquad 
\]
We notice that $k$ is determined by $i$ and $j$, and that the assertion of
the displayed equation is that $k$ has to exist.  Because all the $\tau_j$
are bijections on $\RR^d$, the induced relabeling of the points will be a 
bijection on the labels.  For example, in Figure \ref{fig:desarguestypes}, the 
group $T$ acting on the point configuration given by the vertices of 
the frameworks in (a), (b) is $T = \{\operatorname{Id}, \sigma\}$, 
where $\sigma$ is the mirror reflection in the $y$-axis.  The 
relabelings of the points are different between them, since $\p_5$ and 
$\p_6$ are exchanged in (a), but fixed in (b).  In Figure \ref{fig:desarguestypes}
(c) the symmetry group is $\{\operatorname{id}, C_2\}$, where $C_2$ is the 
half-turn around the origin.  The induced relabeling of the points is, again, 
different, because, e.g., $\p_1$ is mapped to $\p_4$ in (c) and to $\p_2$ in (a),
(b).

\begin{figure}[htp]
\begin{center}
\begin{tikzpicture}[very thick,scale=1.2]
   \tikzstyle{every node}=[circle, draw=black, fill=white, inner sep=0pt, minimum width=4pt];

   \path (-1,-0.5) node (p3)  {} ;
 \path (-1.2,0.5) node (p1)  {} ;
    \path (1,-0.5) node (p4)  {} ;
   \path (1.2,0.5) node (p2)  {} ;
\path (-.5,-0.2) node (l)  {} ;
   \path (0.5,-0.2) node (r)  {} ;
\node [draw=white, fill=white] (b) at (-1.2,0.75) {$\p_1$};
 \node [draw=white, fill=white] (b) at (1.2,0.75) {$\p_2$};
 \node [draw=white, fill=white] (b) at (-1.0,-0.75) {$\p_3$};
 \node [draw=white, fill=white] (b) at (1.0,-0.75) {$\p_4$};
  \node [draw=white, fill=white] (b) at (-0.8,-0.2) {$\p_5$};
 \node [draw=white, fill=white] (b) at (0.8,-0.2) {$\p_6$};

\draw (p1)  --  (p2);
        \draw (p1)  --  (p3);
\draw (p4)  --  (p2);
        \draw (p4)  --  (p3);
\draw (l)  --  (p1);
        \draw (l)  --  (p3);
\draw (r)  --  (p2);
        \draw (r)  --  (p4);
        \draw(l)--(r);

\draw[thin, dashed] (0,-0.8)--(0,0.8);
   
 \node [draw=white, fill=white] (b) at (0,-1.4) {(a)};
\end{tikzpicture}
\hspace{0.2cm}
\begin{tikzpicture}[very thick,scale=1.2]
   \tikzstyle{every node}=[circle, draw=black, fill=white, inner sep=0pt, minimum width=4pt];

   \path (-0.5,0.8) node (p1)  {} ;
 \path (0.5,0.8) node (p2)  {} ;
    \path (-0.7,-0.8) node (p3)  {} ;
   \path (0.7,-0.8) node (p4)  {} ;
\path (0,-0.2) node (l)  {} ;
   \path (0,0.5) node (r)  {} ;

\node [draw=white, fill=white] (b) at (-0.8,0.8) {$\p_1$};
 \node [draw=white, fill=white] (b) at (0.8,0.8) {$\p_2$};
 \node [draw=white, fill=white] (b) at (-1.0,-0.8) {$\p_3$};
 \node [draw=white, fill=white] (b) at (1.0,-0.8) {$\p_4$};
  \node [draw=white, fill=white] (b) at (-0.3,-0.2) {$\p_5$};
 \node [draw=white, fill=white] (b) at (-0.3,0.4) {$\p_6$};

\draw (p1)  --  (p2);
        \draw (p1)  --  (p3);
\draw (p4)  --  (p2);
        \draw (p4)  --  (p3);
\draw (l)  --  (p3);
        \draw (l)  --  (p4);
\draw (r)  --  (p1);
        \draw (r)  --  (p2);
        \draw(l)--(r);

\draw[thin, dashed] (0,-1)--(0,1);
   
 \node [draw=white, fill=white] (b) at (0,-1.4) {(b)};
\end{tikzpicture}
\hspace{0.2cm}
\begin{tikzpicture}[very thick,scale=1.2]
   \tikzstyle{every node}=[circle, draw=black, fill=white, inner sep=0pt, minimum width=4pt];

   \path (-1,-0.5) node (p3)  {} ;
 \path (-1,0.5) node (p1)  {} ;
    \path (1,-0.5) node (p4)  {} ;
   \path (1,0.5) node (p2)  {} ;
\path (-0.5,0.1) node (l)  {} ;
   \path (0.5,-0.1) node (r)  {} ;

\node [draw=white, fill=white] (b) at (-1,0.75) {$\p_1$};
 \node [draw=white, fill=white] (b) at (1,0.75) {$\p_2$};
 \node [draw=white, fill=white] (b) at (-1.0,-0.75) {$\p_3$};
 \node [draw=white, fill=white] (b) at (1.0,-0.75) {$\p_4$};
  \node [draw=white, fill=white] (b) at (-0.8,0) {$\p_5$};
 \node [draw=white, fill=white] (b) at (0.8,-0.1) {$\p_6$};

\draw (p1)  --  (p2);
        \draw (p1)  --  (p3);
\draw (p4)  --  (p2);
        \draw (p4)  --  (p3);
\draw (l)  --  (p1);
        \draw (l)  --  (p3);
\draw (r)  --  (p2);
        \draw (r)  --  (p4);
        \draw(l)--(r);

\node [draw=white, fill=white] (b) at (0,-1.4) {(c)};
\end{tikzpicture}
\hspace{0.2cm}
\begin{tikzpicture}[very thick,scale=1.2]
   \tikzstyle{every node}=[circle, draw=black, fill=white, inner sep=0pt, minimum width=4pt];

   \path (-1,-0.5) node (p3)  {} ;
 \path (-1,0.5) node (p1)  {} ;
    \path (1,-0.5) node (p4)  {} ;
   \path (1,0.5) node (p2)  {} ;
\path (-0.5,0) node (l)  {} ;
   \path (0.5,0) node (r)  {} ;

\node [draw=white, fill=white] (b) at (-1,0.75) {$\p_1$};
 \node [draw=white, fill=white] (b) at (1,0.75) {$\p_2$};
 \node [draw=white, fill=white] (b) at (-1.0,-0.75) {$\p_3$};
 \node [draw=white, fill=white] (b) at (1.0,-0.75) {$\p_4$};
  \node [draw=white, fill=white] (b) at (-0.8,0) {$\p_5$};
 \node [draw=white, fill=white] (b) at (0.8,0) {$\p_6$};

\draw (p1)  --  (p2);
        \draw (p1)  --  (p3);
\draw (p4)  --  (p2);
        \draw (p4)  --  (p3);
\draw (l)  --  (p1);
        \draw (l)  --  (p3);
\draw (r)  --  (p2);
        \draw (r)  --  (p4);
        \draw(l)--(r);
\draw[thin, dashed] (0,-0.8)--(0,0.8);
  \draw[thin, dashed] (-1.5,0)--(1.5,0);
\node [draw=white, fill=white] (b) at (0,-1.4) {(d)};
\end{tikzpicture}
\hspace{0.2cm}
             \begin{tikzpicture}[very thick,scale=1.2]
\tikzstyle{every node}=[circle, draw=black, fill=white, inner sep=0pt, minimum width=4pt];

\node [draw=white, fill=white] (b) at (-0.8,-0.8) {$\p_1$};
 \node [draw=white, fill=white] (b) at (0.8,-0.8) {$\p_2$};
 \node [draw=white, fill=white] (b) at (0,1.3) {$\p_3$};
 \node [draw=white, fill=white] (b) at (-0.38,0) {$\p_4$};
  \node [draw=white, fill=white] (b) at (0.23,-0.35) {$\p_5$};
 \node [draw=white, fill=white] (b) at (0.2,0.36) {$\p_6$};

        \path (90:1cm) node (p1) {} ;
       \path (210:1cm) node (p2) {} ;
           \path (330:1cm) node (p3) {} ;

 \path (100:0.4cm) node (p11) {} ;
       \path (220:0.4cm) node (p22) {} ;
           \path (340:0.4cm) node (p33) {} ;

        \draw (p1)  --  (p2);
         \draw (p1)  --  (p3);
        \draw (p3)  --  (p2);
        \draw (p11)  --  (p22);
         \draw (p11)  --  (p33);
        \draw (p33)  --  (p22);
         \draw (p1)  --  (p11);
         \draw (p2)  --  (p22);
        \draw (p3)  --  (p33);
                
        \node [rectangle,draw=white, fill=white] (e) at (0,-1.4) {(e)};        
                  \end{tikzpicture}
                \hspace{0.2cm}  
             \begin{tikzpicture}[very thick,scale=1.2]
\tikzstyle{every node}=[circle, draw=black, fill=white, inner sep=0pt, minimum width=4pt];

 \node [draw=white, fill=white] (b) at (-0.48,0) {$\p_4$};
  \node [draw=white, fill=white] (b) at (0.2,-0.38) {$\p_5$};
 \node [draw=white, fill=white] (b) at (0.2,0.36) {$\p_6$};

        \path (90:1cm) node (p1) {} ;
       \path (210:1cm) node (p2) {} ;
           \path (330:1cm) node (p3) {} ;

 \path (90:0.5cm) node (p11) {} ;
       \path (210:0.5cm) node (p22) {} ;
           \path (330:0.5cm) node (p33) {} ;

\node [draw=white, fill=white] (b) at (-0.8,-0.8) {$\p_1$};
 \node [draw=white, fill=white] (b) at (0.8,-0.8) {$\p_2$};
 \node [draw=white, fill=white] (b) at (0,1.3) {$\p_3$};
           
        \draw (p1)  --  (p2);
         \draw (p1)  --  (p3);
        \draw (p3)  --  (p2);
        \draw (p11)  --  (p22);
         \draw (p11)  --  (p33);
        \draw (p33)  --  (p22);
         \draw (p1)  --  (p11);
         \draw (p2)  --  (p22);
        \draw (p3)  --  (p33);

\draw[thin, dashed] (0,-0.8)--(0,0.8);
\draw[thin, dashed] (210:0.8cm)--(30:0.8cm);
     \draw[thin, dashed] (150:0.8cm)--(330:0.8cm);           
         \node [rectangle,draw=white, fill=white] (a) at (0,-1.4) {(f)};        
                  \end{tikzpicture}
       
\end{center}
\vspace{-0.6cm} \caption{The triangular prism graph can be realised as a bar-joint framework in the plane with reflection symmetry $\mathcal{C}_s$ (a,b), half-turn symmetry $\mathcal{C}_2$ (c), dihedral symmetry $\mathcal{C}_{2v}$ of order 4 (d), three-fold rotational symmetry $\mathcal{C}_3$ (e) and dihedral symmetry $\mathcal{C}_{3v}$ of order 6 (f).  The symmetry types shown here correspond to the choices of the 
combinatorial group $\Gamma$ and point group representation $\tau$ that lead to 
planar realisations.  Other choices exist, but the associated drawings have 
crossings.}
\label{fig:desarguestypes}
\end{figure}
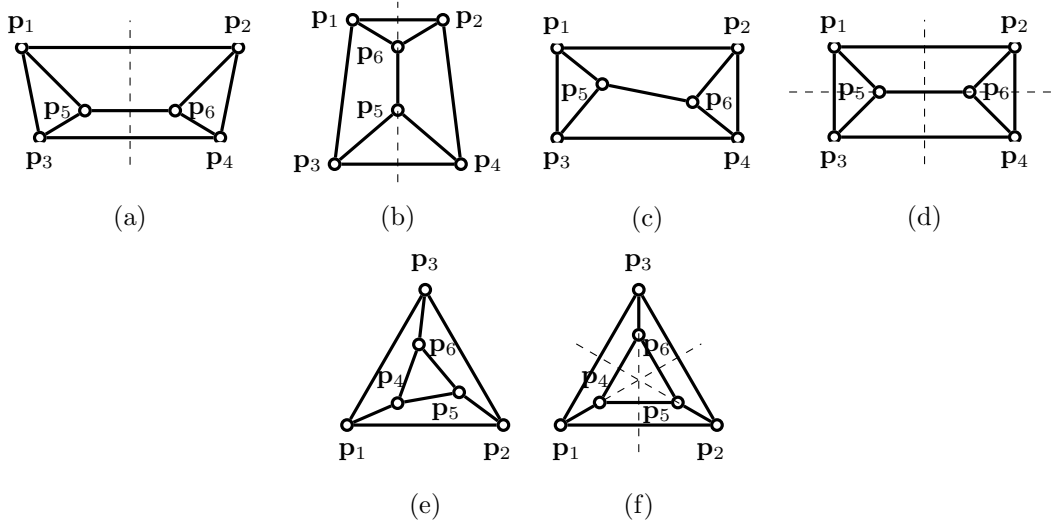

Finally, we combine the two concepts to say what a symmetric framework is.  The 
notation is $(G,\Gamma,\tau,\p)$ which we call a \emph{$(\Gamma,\tau)$-symmetric 
framework}. We sometimes also just write $(G,\p)$ in short. The data are $\Gamma < \Aut(G)$, a subgroup of the combinatorial symmetry group of $G$, and $\tau : \Gamma\to \Euc_0(d)$ a faithful representation of 
$\Gamma$ by origin-fixing isometries.  This means that, if 
\[
    \Gamma = \{\gamma_1, \ldots, \gamma_g\}\qquad \text{and}\qquad 
    T = \{\tau(\gamma_1), \ldots, \tau({\gamma_g})\}
\]
then $T=\tau(\Gamma)$ is a finite subgroup of isometries that 
has the same structure as $\Gamma$.  In symbols, we have, for all $i,j$
\[
    \tau(\gamma_i\gamma_j) = \tau(\gamma_i)\tau(\gamma_j)
\]
and also that
\[
    \tau(\gamma) = \Id\qquad \textrm{ implies that }\qquad \gamma = \Id
\]
Notice that this puts a lot of restrictions on what $\Gamma$ can be, since there 
are not very many finite subgroups of $\Euc_0(d)$, at least when $d$ is 
small relative to $n$.  To 
complete the definition, we add a compatibility condition between 
what $\Gamma$ does to $G$ and what $T$ does to $\p$ (it is not 
enough to just say that $\p$ is $T$-symmetric in some way).  What we need is 
that 
\[
    \tau(\gamma)\p_i = \p_{\gamma(i)}\qquad \text{for all $\gamma\in \Gamma$ and 
    $i\in V$}
\]
In words, this says that, for each element $\gamma$ of $\Gamma$  and 
point $\p_i$, if we apply the geometric operation $\tau(\gamma)$
to the point $\p_i$ we get the point $\p_{\gamma(i)}$ that 
corresponds to how $\gamma$ moves the vertex $i$. We can now 
consider Figure \ref{fig:desarguestypes} in more detail.  In 
Figure \ref{fig:desarguestypes}(a), $\Gamma$ is the group 
$\{\operatorname{Id}, (1\, 2)(5\, 6)(3\, 4)\}$, and 
$\tau(\Gamma)$ is the group $\{\operatorname{Id}, \sigma\}$, as described 
above.  The compatibility conditions are quickly checked to be met.
In Figure \ref{fig:desarguestypes}(b), $\tau(\Gamma)$ is the same group as 
in (a), but $\Gamma$ is different, namely 
$\{\operatorname{Id}, (1\, 2)(3\, 4)\}$; this shows why the compatibility 
condition asks for more than $\p$ being $T$-symmetric in some way.
Figure \ref{fig:desarguestypes}(c) is an example where $\Gamma$ and 
$\tau(\Gamma)$ are both different from what we have seen in (a) and (b): 
$\Gamma = \{\operatorname{Id},(1\, 4)(2\, 3)(5\, 6)\}$ and 
$\tau(\Gamma) = \{\operatorname{Id},C_2\}$.

Note that for any $2$-dimensional framework, $\tau(\gamma)$ is either the identity, denoted by  $\textrm{Id}$, a rotation by $\frac{2\pi}{n}$, $n\in \mathbb{N}$, about the origin, denoted by $C_n$, or a  reflection in a line through the origin,  denoted by $\sigma$. The point groups $\tau(\Gamma)$ that can be created from these operations are the infinite sets $\mathcal{C}_n$ and $\mathcal{C}_{nv}$ for all $n\in \mathbb{N}$. The group $\mathcal{C}_n$ is the cyclic group generated by $C_n$, and $\mathcal{C}_{nv}$ is the dihedral group generated by a pair $\{C_n,\sigma\}$. The reflection group $\mathcal{C}_{1v}$ is usually denoted by $\mathcal{C}_s$. See \cite{atk70,alt94}, for example, for further details and a corresponding classification  in $3$-space.

We conclude this section with Figure \ref{fig:2}, which shows additional 
examples of $(\Gamma, \tau)$-symmetric frameworks on the same underlying 
graphs as $\Gamma$ and $\tau$ vary.  Since the methods used to analyse 
symmetric frameworks need $\Gamma$ and $\tau$ fixed before they can be 
used, examples like these show that these early choices constrain what 
kinds of frameworks are possible.  We will also see later that the different 
choices of $\Gamma$ and $\tau$ can give rise to different rigidity properties and not 
just look different.

\begin{figure}[htp]
\begin{center}
\begin{tikzpicture}[very thick,scale=1]
   \tikzstyle{every node}=[circle, draw=black, fill=white, inner sep=0pt, minimum width=4pt];

   \path (-1,-0.5) node (p3)  {} ;
 \path (-1.2,0.5) node (p1)  {} ;
    \path (1,-0.5) node (p4)  {} ;
   \path (1.2,0.5) node (p2)  {} ;
\path (-.5,0) node (l)  {} ;
   \path (0.5,0) node (r)  {} ;

 \node [draw=white, fill=white] (b) at (-1.5,0.6) {$\p_1$};
 \node [draw=white, fill=white] (b) at (1.5,0.6) {$\p_2$};
 \node [draw=white, fill=white] (b) at (-1.3,-0.6) {$\p_3$};
 \node [draw=white, fill=white] (b) at (1.3,-0.6) {$\p_4$};
  \node [draw=white, fill=white] (b) at (-0.8,0) {$\p_5$};
 \node [draw=white, fill=white] (b) at (0.8,0) {$\p_6$};

\draw (p1)  --  (p2);
        \draw (p1)  --  (p3);
\draw (p4)  --  (p2);
        \draw (p4)  --  (p3);
\draw (l)  --  (p1);
        \draw (l)  --  (p3);
\draw (r)  --  (p2);
        \draw (r)  --  (p4);
        \draw(l)--(r);

\draw[thin, dashed] (0,-0.8)--(0,0.8);
   
 \node [draw=white, fill=white] (b) at (0,-1.5) {(a)};
\end{tikzpicture}
\hspace{0.3cm}
       \begin{tikzpicture}[very thick,scale=1]
   \tikzstyle{every node}=[circle, draw=black, fill=white, inner sep=0pt, minimum width=4pt];

   \path (-1.2,-0.5) node (p3)  {} ;
 \path (-1.2,0.7) node (p1)  {} ;
    \path (1.2,-0.7) node (p4)  {} ;
   \path (1.2,0.5) node (p2)  {} ;
\path (-.5,0.1) node (l)  {} ;
   \path (0.5,-0.1) node (r)  {} ;

 \node [draw=white, fill=white] (b) at (-1.5,0.6) {$\p_1$};
 \node [draw=white, fill=white] (b) at (1.5,0.6) {$\p_4$};
 \node [draw=white, fill=white] (b) at (-1.5,-0.6) {$\p_3$};
 \node [draw=white, fill=white] (b) at (1.5,-0.6) {$\p_2$};
  \node [draw=white, fill=white] (b) at (-0.8,0.1) {$\p_5$};
 \node [draw=white, fill=white] (b) at (0.8,-0.1) {$\p_6$};

\draw (p1)  --  (p4);
        \draw (p1)  --  (p3);
\draw (p4)  --  (p2);
        \draw (p2)  --  (p3);
\draw (l)  --  (p1);
        \draw (l)  --  (p3);
\draw (r)  --  (p2);
        \draw (r)  --  (p4);
        \draw(l)--(r);

   
 \node [draw=white, fill=white] (b) at (0,-1.5) {(b)};
\end{tikzpicture}
\hspace{0.3cm}
\begin{tikzpicture}[very thick,scale=1]
   \tikzstyle{every node}=[circle, draw=black, fill=white, inner sep=0pt, minimum width=4pt];

   \path (-1,-0.5) node (p3)  {} ;
 \path (-1.2,0.5) node (p1)  {} ;
    \path (1,-0.5) node (p4)  {} ;
   \path (1.2,0.5) node (p2)  {} ;
\path (-.5,0) node (l)  {} ;
   \path (0.5,0) node (r)  {} ;

 \node [draw=white, fill=white] (b) at (-1.5,0.6) {$\p_1$};
 \node [draw=white, fill=white] (b) at (1.5,0.6) {$\p_4$};
 \node [draw=white, fill=white] (b) at (-1.3,-0.6) {$\p_3$};
 \node [draw=white, fill=white] (b) at (1.3,-0.6) {$\p_2$};
  \node [draw=white, fill=white] (b) at (-0.8,0) {$\p_5$};
 \node [draw=white, fill=white] (b) at (0.8,0) {$\p_6$};

\draw (p1)  --  (p4);
        \draw (p1)  --  (p3);
\draw (p4)  --  (p2);
        \draw (p2)  --  (p3);
\draw (l)  --  (p1);
        \draw (l)  --  (p3);
\draw (r)  --  (p2);
        \draw (r)  --  (p4);
        \draw(l)--(r);

\draw[thin, dashed] (0,-0.8)--(0,0.8);
   
 \node [draw=white, fill=white] (b) at (0,-1.5) {(c)};
\end{tikzpicture}
\hspace{0.3cm}
       \begin{tikzpicture}[very thick,scale=1]
   \tikzstyle{every node}=[circle, draw=black, fill=white, inner sep=0pt, minimum width=4pt];

   \path (-1.2,-0.5) node (p3)  {} ;
 \path (-1.2,0.7) node (p1)  {} ;
    \path (1.2,-0.7) node (p4)  {} ;
   \path (1.2,0.5) node (p2)  {} ;
\path (-.5,0.1) node (l)  {} ;
   \path (0.5,-0.1) node (r)  {} ;

 \node [draw=white, fill=white] (b) at (-1.5,0.6) {$\p_1$};
 \node [draw=white, fill=white] (b) at (1.5,0.6) {$\p_2$};
 \node [draw=white, fill=white] (b) at (-1.5,-0.6) {$\p_3$};
 \node [draw=white, fill=white] (b) at (1.5,-0.6) {$\p_4$};
  \node [draw=white, fill=white] (b) at (-0.8,0.1) {$\p_5$};
 \node [draw=white, fill=white] (b) at (0.8,-0.1) {$\p_6$};

\draw (p1)  --  (p2);
        \draw (p1)  --  (p3);
\draw (p4)  --  (p2);
        \draw (p4)  --  (p3);
\draw (l)  --  (p1);
        \draw (l)  --  (p3);
\draw (r)  --  (p2);
        \draw (r)  --  (p4);
        \draw(l)--(r);

   
 \node [draw=white, fill=white] (b) at (0,-1.5) {(d)};
\end{tikzpicture}
\end{center}
\vspace{-0.6cm} \caption{$(\Gamma,\tau)$-symmetric frameworks in the plane, where $\Gamma$ has order $2$. For (a) and (b), $\Gamma=\{\Id, (1,2)(3,4)(5,6)\}$. However, in (a) the non-trivial element of $\Gamma$ is mapped to a reflection, whereas in (b) it is mapped to the half turn. For (c) and (d), $\Gamma=\{\Id, (1,4)(2,3)(5,6)\}$, where in (c) and (d) the non-trivial element of $\Gamma$ is again mapped to a reflection and the half turn, respectively. }
\label{fig:2}
\end{figure}
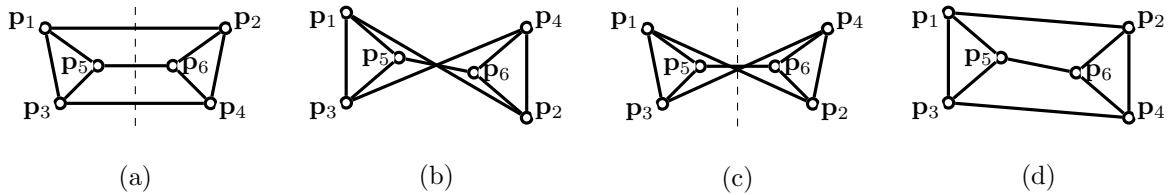

Now that we have the formal set-up, we make a first observation: 

\medskip
 \noindent   \emph{If we want to know whether $G$ has ``some symmetric realisation'' without 
a ``combinatorial symmetry type'' $\Gamma$ in mind, then we first have to 
compute  $\Aut(G)$, enumerate its subgroups, and then identify the ones that 
have a faithful representation $\tau : \Gamma \to \Euc_0(d)$.}
\medskip

In general, this is a difficult computational problem, but it is tractable for 
small and medium-sized examples using state-of-the art computational algebra
systems, such as Magma, GAP, Oscar, and Mathematica.  These incorporate 
optimised heuristic algorithms for computing $\Aut(G)$  and for identifying 
point groups. 
When the graph is informed through engineering or design decisions, this can often be done through observation as the desired symmetry group is well known by the practitioner.

\section{Background on rigidity theory}\label{sec:rig}

An \emph{infinitesimal motion} of a framework $(G,\p)$  in $\mathbb{R}^d$ is a function $\mathbf{u}: V\to \mathbb{R}^{d}$ such that
\begin{equation}
\label{infinmotioneq}
(\p_i-\p_j)\cdot (\mathbf u_i-\mathbf u_j) =0 \quad\textrm{ for all } ij \in  E\textrm{,}
\end{equation}
where $\p_i=\p(i)$, $\mathbf u_i=\mathbf u(i)$ for each $i$ and the $\cdot$ symbol denotes the standard inner product on $\mathbb{R}^d$. We note that an infinitesimal motion can be thought of as an element of the space $\mathbb{R}^{dn}$.

An infinitesimal motion $\mathbf u$ of  $(G,\p)$ is  \emph{trivial} if there exists a skew-symmetric matrix $S$
and a vector $\mathbf t$ such that $\mathbf u_i=S \p_i+\mathbf t$ for all $i\in V$, i.e., if $\mathbf u$ corresponds to a rigid body motion in the plane. A non-trivial infinitesimal motion is also called an \emph{infinitesimal flex}.

The matrix corresponding to the linear system in (\ref{infinmotioneq})  with the $\mathbf u_i$ being the unknowns, is the \emph{rigidity matrix} (or \emph{equilibrium matrix}), denoted $R(\p)$.  If the points $\p_i$ affinely span $\RR^d$, the infinitesimal rigidity is 
equivalent to the rank of $R(\p)$ being $dn - \binom{d+1}{2}$.  (When a framework 
has few vertices, it will be infinitesimally rigid if and only if $G$ 
is complete and $\p$ is affinely independent.  In applications, this 
situation rarely arises.)

A \emph{self-stress} of a framework $(G,\p)$ is a function  $\mathbf{\omega}:E\to \mathbb{R}$ such that for each  vertex $i$ of $G$ the following vector equation holds:
\begin{displaymath}
\sum_{j :ij\in E}\mathbf \omega(ij)(\p_{i}-\p_{j})=0 \textrm{.}
\end{displaymath}
In structural engineering, $\mathbf \omega(ij)(\p_{i}-\p_{j})$ is the signed axial force in the bar $ij$
and the stress-coefficient $\mathbf \omega(ij)$ is called the force-density (scalar force divided
by the bar length) of the bar $ij$.
The summation above says that the tensions and compressions in the bars balance at each node $i$, and hence a self-stress is also known as an \emph{equilibrium stress}.
For the engineer, a self-stress is often considered as a set of axial forces within a framework which are in equilibrium in the absence of external loads. 
Note that $\mathbf \omega \in \mathbb{R}^{|E|}$ is a self-stress of $(G,\p)$ if and only if $\mathbf \omega^TR(\p)=0$ (i.e. it lies in the left null space of $R(\p)$). A framework is called \emph{independent} if it has no non-zero self-stress, and it is called \emph{isostatic} if it is infinitesimally rigid and independent.
The term $\mathbf \omega$ is called the \textit{force density} by engineers and the \textit{stress} by rigidity theorists -- please refer to Connelly and Guest \cite{CGbook} for further details on the engineering and mathematical terms.

A framework $(G,\p)$ in $\mathbb{R}^d$ is \emph{generic} if $\textrm{rank}R(\p)=\textrm{max}\{\textrm{rank} R(\p')|\, \p'\in \mathbb{R}^{d|V|} \}$. It is easy to see that the set of configurations $\p$ of generic frameworks $(G,\p)$ in $\mathbb{R}^d$ forms an open dense  subset of $\mathbb{R}^{d|V|}$. Moreover, all generic $d$-dimensional frameworks on a graph $G$ share the same infinitesimal rigidity properties, i.e. they are either all infinitesimally rigid, or none of them are. Thus, we  say that a graph $G$ is \emph{$d$-rigid} (\emph{$d$-isostatic, $d$-independent}) if any (or equivalently every) generic realisation of $G$ as a framework in $\mathbb{R}^d$ is infinitesimally rigid (isostatic, independent).

\section{Symmetric averaging}\label{sec:symav}
If we want to understand the effect of symmetry on rigidity 
properties, such as the number of independent self-stresses, a starting point 
is to ask when a given framework $(G,\p)$ has a 
nearby $(\Gamma,\tau)$-symmetric framework.  As we have already seen,
a rigidity analysis depends on the $(\Gamma,\tau)$ data, so we need 
to make that choice in advance.  Once we have, a natural next step 
for a local search procedure is to find the closest $(\Gamma,\tau)$-symmetric
framework to $(G,\p)$.

A first proposal draws inspiration from the classic averaging technique of 
Whiteley \cite{bernd2017}.  This takes two  frameworks 
$(G,\p)$ and $(G,\q)$ that have the same edge lengths and 
produces a framework $(G,\frac{1}{2}(\p + \q))$.  Whiteley 
showed that if the pair of original frameworks are non-congruent, then 
the averaged framework has a non-trivial infinitesimal motion.  By 
the index theorem, if $(G,\p)$ and $(G,\q)$  infinitesimally rigid, then 
the averaged framework must have an unexpected self-stress.
Let us specialise our setting even more, and assume that $\Gamma$
has order two and $\tau(\Gamma) = \{\operatorname{Id},\sigma\}$ 
consists of the identity and a single mirror reflection.  The proposal 
would be to take the average of the frameworks $(G,\p)$ and 
$(G,\sigma\p)$ (we apply $\sigma$ to each $\p_i$).
 Unfortunately, this procedure isn't useful, 
because the average of $\p_i$ and $\sigma(\p_i)$ must lie on the 
mirror line. We will always get a highly degenerate framework this way.
Instead, we will describe a method to find the closest 
$(\Gamma,\tau)$-symmetric framework to $(G,\p)$ via orthogonal 
projection.

\subsection{Symmetric averaging for the reflection group}

We stick with the example from above. Let the 
given framework $(G,\p)$ have vertex set $\{1, \ldots, n\}$ and an involution 
$\gamma\in \Aut(G) < \Sym(n)$.  As before, $\Gamma = \{\operatorname{Id},\gamma\}$
and $\tau(\gamma) = \sigma$.
We define 
\[
    (\gamma\cdot \p) \qquad \text{by the relation} \qquad (\gamma\cdot \p)_i = \sigma(\p_{\gamma(i)})
\]
which is different from what we had before, because we use the 
permutation $\gamma$ and the reflection $\sigma$ together.
The relabeling is formal, but it is motivated by the idea that, if $(G,\p)$ was 
already symmetric with respect to the reflection $\sigma$, then $\sigma(\p_{\gamma(i)})$ would equal $\p_i$, for all $i$. Since 
the ``symmetrically averaged''  framework $(G, \frac{1}{2}(\p + \gamma\cdot \p))$ 
is, in fact, $(\Gamma,\tau)$-symmetric (we will check this soon), 
this procedure has two good properties: its output is symmetric; it fixes symmetric 
frameworks.  Moreover, unlike the more naive proposal, it does not 
collapse its input.  Interestingly, this ``relabel, reflect, average'' 
procedure can have a surprising number of 
behaviours with respect to infinitesimal flexes and self-stresses, 
as we will discuss in Section~\ref{sec:stresses}.
The symmetric averaging procedure is illustrated in Figure~\ref{fig:avgex} for the graph of the triangular prism. Here $\Gamma<\textrm{Aut}(G)$ is of order $2$ and generated by $\gamma=(1,2)(3)(4,5)(6)$.

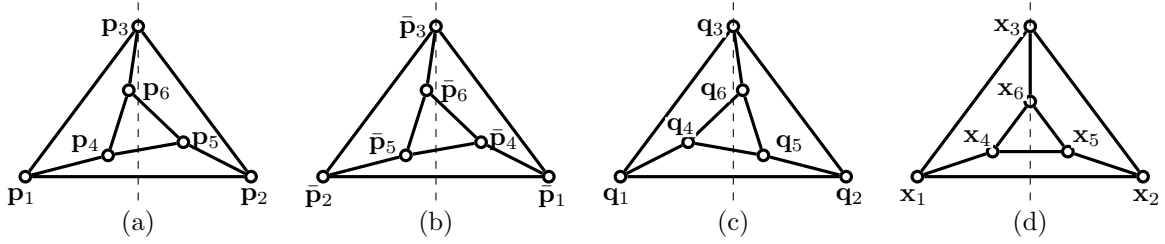
\begin{figure}[htp]
\begin{center}
   \begin{tikzpicture}[very thick,scale=0.5]
\tikzstyle{every node}=[circle, draw=black, fill=white, inner sep=0pt, minimum width=4pt];

\node [rectangle,draw=white, fill=white] (a) at (-3.1,-1.2) {$\p_1$}; 
\node [rectangle,draw=white, fill=white] (a) at (3.1,-1.2) {$\p_2$};
\node [rectangle,draw=white, fill=white] (a) at (-0.6,3.3) {$\p_3$};
\node [rectangle,draw=white, fill=white] (a) at (0.5,1.6) {$\p_6$};
\node [rectangle,draw=white, fill=white] (a) at (-1.4,0.2) {$\p_4$};
\node [rectangle,draw=white, fill=white] (a) at (1.8,0.3) {$\p_5$};

        \path (-0.8,-0.1) node (p1) {} ;
        \path (1.2,0.25) node (p2) {} ;
        \path (-0.25,1.63333) node (p3) {} ;
                
         \path (-3,-0.66666) node (p11) {} ;
        \path (3,-0.66666) node (p22) {} ;
        \path (0,3.333333) node (p33) {} ;
           
        \draw (p1)  --  (p2);
         \draw (p1)  --  (p3);
        \draw (p3)  --  (p2);
        \draw (p11)  --  (p22);
         \draw (p11)  --  (p33);
        \draw (p33)  --  (p22);
         \draw (p1)  --  (p11);
         \draw (p2)  --  (p22);
        \draw (p3)  --  (p33);
\draw[thin, dashed] (0,-1.3)--(0,4);

         \node [rectangle,draw=white, fill=white] (a) at (0,-1.9) {(a)};        
                  \end{tikzpicture}
                  \hspace{0.2cm}
              \begin{tikzpicture}[very thick,scale=0.5]
\tikzstyle{every node}=[circle, draw=black, fill=white, inner sep=0pt, minimum width=4pt];

\node [rectangle,draw=white, fill=white] (a) at (-3.1,-1.2) {$\bar \p_2$}; 
\node [rectangle,draw=white, fill=white] (a) at (3.1,-1.2) {$\bar \p_1$};
\node [rectangle,draw=white, fill=white] (a) at (-0.6,3.3) {$\bar \p_3$};
\node [rectangle,draw=white, fill=white] (a) at (0.5,1.6) {$\bar \p_6$};
\node [rectangle,draw=white, fill=white] (a) at (-1.4,0.2) {$\bar \p_5$};
\node [rectangle,draw=white, fill=white] (a) at (1.8,0.3) {$\bar \p_4$};

        \path (-0.8,-0.1) node (p1) {} ;
        \path (1.2,0.25) node (p2) {} ;
        \path (-0.25,1.63333) node (p3) {} ;
                
         \path (-3,-0.66666) node (p11) {} ;
        \path (3,-0.66666) node (p22) {} ;
        \path (0,3.333333) node (p33) {} ;
           
        \draw (p1)  --  (p2);
         \draw (p1)  --  (p3);
        \draw (p3)  --  (p2);
        \draw (p11)  --  (p22);
         \draw (p11)  --  (p33);
        \draw (p33)  --  (p22);
         \draw (p1)  --  (p11);
         \draw (p2)  --  (p22);
        \draw (p3)  --  (p33);
\draw[thin, dashed] (0,-1.3)--(0,4);

         \node [rectangle,draw=white, fill=white] (a) at (0,-1.9) {(b)};      
                  \end{tikzpicture}
                              \hspace{0.2cm}
             \begin{tikzpicture}[very thick,scale=0.5]
\tikzstyle{every node}=[circle, draw=black, fill=white, inner sep=0pt, minimum width=4pt];
        \path (0.8,-0.1) node (p2) {} ;
        \path (-1.2,0.25) node (p1) {} ;
        \path (0.25,1.63333) node (p3) {} ;
                
         \path (-3,-0.66666) node (p11) {} ;
        \path (3,-0.66666) node (p22) {} ;
        \path (0,3.333333) node (p33) {} ;
           
        \draw (p1)  --  (p2);
         \draw (p1)  --  (p3);
        \draw (p3)  --  (p2);
        \draw (p11)  --  (p22);
         \draw (p11)  --  (p33);
        \draw (p33)  --  (p22);
         \draw (p1)  --  (p11);
         \draw (p2)  --  (p22);
        \draw (p3)  --  (p33);
\draw[thin, dashed] (0,-1.3)--(0,4);

\node [rectangle,draw=white, fill=white] (a) at (-3.1,-1.2) {$\q_1$}; 
\node [rectangle,draw=white, fill=white] (a) at (3.1,-1.2) {$\q_2$};
\node [rectangle,draw=white, fill=white] (a) at (-0.6,3.3) {$\q_3$};
\node [rectangle,draw=white, fill=white] (a) at (-0.5,1.6) {$\q_6$};
\node [rectangle,draw=white, fill=white] (a) at (-1.4,0.65) {$\q_4$};
\node [rectangle,draw=white, fill=white] (a) at (1.5,0.2) {$\q_5$};

         \node [rectangle,draw=white, fill=white] (a) at (0,-1.9) {(c)};        
                  \end{tikzpicture}
                  \hspace{0.2cm}        
  \begin{tikzpicture}[very thick,scale=0.5]
\tikzstyle{every node}=[circle, draw=black, fill=white, inner sep=0pt, minimum width=4pt];

        \path (-1,0) node (p1) {} ;
        \path (1,0) node (p2) {} ;
        \path (0,1.33333) node (p3) {} ;
        
           \node [draw=white, fill=white] (a) at (-0.5,1.7) {};
          \node [draw=white, fill=white] (a) at (-1,-1.1) {};  
         \node [draw=white, fill=white] (a) at (1.7,0.2) {};
               
         \path (-3,-0.66666) node (p11) {} ;
        \path (3,-0.66666) node (p22) {} ;
        \path (0,3.333333) node (p33) {} ;

        \draw (p1)  --  (p2);
         \draw (p1)  --  (p3);
        \draw (p3)  --  (p2);
        \draw (p11)  --  (p22);
         \draw (p11)  --  (p33);
        \draw (p33)  --  (p22);
         \draw (p1)  --  (p11);
         \draw (p2)  --  (p22);
        \draw (p3)  --  (p33);
        \draw[thin, dashed] (0,-1.3)--(0,4);
             

\node [rectangle,draw=white, fill=white] (a) at (-3.1,-1.2) {$\x_1$}; 
\node [rectangle,draw=white, fill=white] (a) at (3.1,-1.2) {$\x_2$};
\node [rectangle,draw=white, fill=white] (a) at (-0.6,3.3) {$\x_3$};
\node [rectangle,draw=white, fill=white] (a) at (-0.5,1.6) {$\x_6$};
\node [rectangle,draw=white, fill=white] (a) at (-1.4,0.35) {$\x_4$};
\node [rectangle,draw=white, fill=white] (a) at (1.5,0.35) {$\x_5$};

            \node [rectangle,draw=white, fill=white] (a) at (0,-1.9) {(d)};
      \end{tikzpicture}

                \end{center}
\caption{Generating a symmetrically averaged framework $(G,\x)$ from a non-symmetric framework $(G,\p)$ for the case of the reflection group $\mathcal{C}_s$ in the plane: $(G,\bar\p)$ in (b) is obtained from $(G,\p)$ in (a) by relabelling. $(G,\q)$ in (c) is obtained from $(G,\bar\p)$ in (b) by reflecting in the vertical mirror line, so that $\q=(\gamma \cdot \p)$. Finally,  $(G,\x)$ in (d) is the average of (a) and (c) and has mirror symmetry. }
\label{fig:avgex}
\end{figure}

\subsection{Symmetric averaging in general}

Now we generalise the symmetric averaging procedure to any 
group.  We begin with a framework $(G,\p)$, not necessarily 
symmetric.  Since our target is a symmetric framework, we 
need to fix the data $(\Gamma,\tau)$, where $\Gamma < \Aut(G)$
is a subgroup of the combinatorial symmetry group of $G$ that admits a faithful
representation $\tau : \Gamma\to \Euc_0(d)$.  
Let $\gamma\in \Gamma$ be given.  We define a new configuration 
$\gamma\cdot \p$ by the relation 
\[
    (\gamma\cdot\p)_i = \tau(\gamma^{-1})\p_{\gamma(i)}
    \qquad \text{for all $i\in V$}
\]
The geometric intuition underlying the definition is from the reflection 
example, and also the idea that, if $\p$ was already $(\Gamma,\tau)$-symmetric,
then $\gamma\cdot \p$ should be equal to $\p$.  This is indeed the case, because 
the compatibility relation $\tau(\gamma)\p_i = \p_{\gamma(i)}$ implies
\[
    (\gamma\cdot\p)_i = 
    \tau(\gamma^{-1})\tau(\gamma)\p_i = \p_i  
\]
With this, we are ready to define the symmetric averaging operator $A$
for a general pair $(\Gamma,\tau)$.  It is defined by the relation 
\[
    A\p = \frac{1}{|\Gamma|}\sum_{\gamma\in \Gamma} \gamma\cdot \p
\]

We illustrate the general symmetric averaging procedure with the example shown in Figure~\ref{fig:avgexgen}. Here $\Gamma<\textrm{Aut}(G)$ is isomorphic to $\mathbb{Z}_3$ and generated by the automorphism $(1,2,3)(4,5,6)$. Moreover, $\tau(\Gamma)$ is the three-fold rotational group $\mathcal{C}_3$ generated by a rotation about the origin by $2\pi/3$ in counterclockwise direction. 

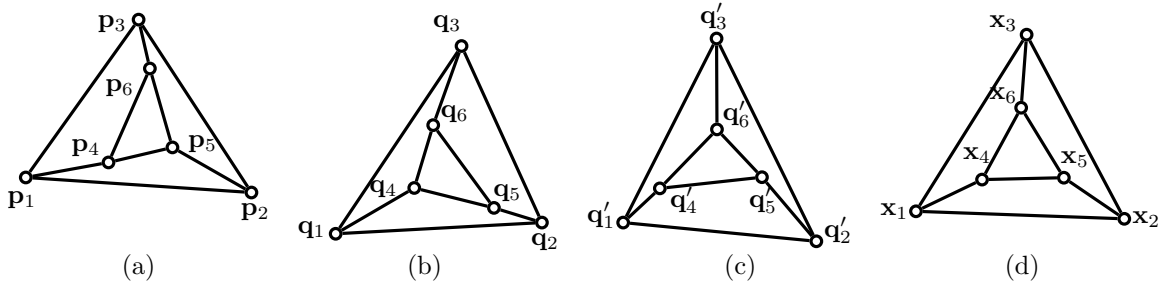
\begin{figure}[htp]
\begin{center}
   \begin{tikzpicture}[very thick,scale=0.5]
\tikzstyle{every node}=[circle, draw=black, fill=white, inner sep=0pt, minimum width=4pt];

\node [rectangle,draw=white, fill=white] (a) at (-3.1,-1.2) {$\p_1$}; 
\node [rectangle,draw=white, fill=white] (a) at (3.1,-1.5) {$\p_2$};
\node [rectangle,draw=white, fill=white] (a) at (-0.7,3.5) {$\p_3$};
\node [rectangle,draw=white, fill=white] (a) at (-0.5,1.8) {$\p_6$};
\node [rectangle,draw=white, fill=white] (a) at (-1.4,0.2) {$\p_4$};
\node [rectangle,draw=white, fill=white] (a) at (1.7,0.3) {$\p_5$};

        \path (-0.8,-0.2) node (p4) {} ;
        \path (0.9,0.2) node (p5) {} ;
        \path (0.3,2.3) node (p6) {} ;
                
         \path (-3,-0.6) node (p1) {} ;
        \path (3,-1) node (p2) {} ;
        \path (0,3.6) node (p3) {} ;
           
        \draw (p4)  --  (p5);
         \draw (p4)  --  (p6);
        \draw (p6)  --  (p5);
        \draw (p1)  --  (p2);
         \draw (p1)  --  (p3);
        \draw (p3)  --  (p2);
         \draw (p4)  --  (p1);
         \draw (p5)  --  (p2);
        \draw (p6)  --  (p3);
                
         \node [rectangle,draw=white, fill=white] (a) at (0,-3) {(a)};        
                  \end{tikzpicture}
                  \hspace{0.1cm}
              \begin{tikzpicture}[very thick,scale=0.5]
\tikzstyle{every node}=[circle, draw=black, fill=white, inner sep=0pt, minimum width=4pt];

\node [rectangle,draw=white, fill=white] (a) at (-3,-2) {$\q_1$}; 
\node [rectangle,draw=white, fill=white] (a) at (3.2,-2.3) {$\q_2$};
\node [rectangle,draw=white, fill=white] (a) at (0.6,3.5) {$\q_3$};
\node [rectangle,draw=white, fill=white] (a) at (0.8,1.1) {$\q_6$};
\node [rectangle,draw=white, fill=white] (a) at (-1.1,-0.9) {$\q_4$};
\node [rectangle,draw=white, fill=white] (a) at (2.2,-1) {$\q_5$};

        \path (0.2268,0.7928) node (p4) {} ;
        \path (-0.2768,-0.879) node (p5) {} ;
        \path (1.8418,-1.409) node (p6) {} ;
                
         \path (0.9804,2.898) node (p1) {} ;
        \path (-2.366,-2.098) node (p2) {} ;
        \path (3.1176,-1.8) node (p3) {} ;
           
        \draw (p4)  --  (p5);
         \draw (p4)  --  (p6);
        \draw (p6)  --  (p5);
        \draw (p1)  --  (p2);
         \draw (p1)  --  (p3);
        \draw (p3)  --  (p2);
         \draw (p4)  --  (p1);
         \draw (p5)  --  (p2);
        \draw (p6)  --  (p3);
                
         \node [rectangle,draw=white, fill=white] (a) at (0,-3) {(b)};        
                  \end{tikzpicture}  
\hspace{0.1cm}
              \begin{tikzpicture}[very thick,scale=0.5]
\tikzstyle{every node}=[circle, draw=black, fill=white, inner sep=0pt, minimum width=4pt];

\node [rectangle,draw=white, fill=white] (a) at (-3.7,-1.6) {$\q'_1$}; 
\node [rectangle,draw=white, fill=white] (a) at (2.6,-2) {$\q'_2$};
\node [rectangle,draw=white, fill=white] (a) at (-0.7,3.7) {$\q'_3$};
\node [rectangle,draw=white, fill=white] (a) at (-0.1,1.1) {$\q'_6$};
\node [rectangle,draw=white, fill=white] (a) at (-1.5,-1.2) {$\q'_4$};
\node [rectangle,draw=white, fill=white] (a) at (0.6,-1.2) {$\q'_5$};

        \path (0.573,-0.5928) node (p4) {} ;
        \path (-0.623,0.679) node (p5) {} ;
        \path (-2.1418,-0.89) node (p6) {} ;
                
         \path (2.019,-2.298) node (p1) {} ;
        \path (-0.634,3.098) node (p2) {} ;
        \path (-3.117,-1.8) node (p3) {} ;
           
        \draw (p4)  --  (p5);
         \draw (p4)  --  (p6);
        \draw (p6)  --  (p5);
        \draw (p1)  --  (p2);
         \draw (p1)  --  (p3);
        \draw (p3)  --  (p2);
         \draw (p4)  --  (p1);
         \draw (p5)  --  (p2);
        \draw (p6)  --  (p3);
                
         \node [rectangle,draw=white, fill=white] (a) at (0,-3) {(c)};        
                  \end{tikzpicture}  
                \hspace{0.1cm}
              \begin{tikzpicture}[very thick,scale=0.5]
\tikzstyle{every node}=[circle, draw=black, fill=white, inner sep=0pt, minimum width=4pt];

\node [rectangle,draw=white, fill=white] (a) at (-3.4,-1.5) {$\x_1$}; 
\node [rectangle,draw=white, fill=white] (a) at (3.3,-1.7) {$\x_2$};
\node [rectangle,draw=white, fill=white] (a) at (-0.5,3.5) {$\x_3$};
\node [rectangle,draw=white, fill=white] (a) at (-0.5,1.6) {$\x_6$};
\node [rectangle,draw=white, fill=white] (a) at (-1.2,0) {$\x_4$};
\node [rectangle,draw=white, fill=white] (a) at (1.4,-0.1) {$\x_5$};

         \path (-2.827,-1.5) node (p1) {} ;
         \path (2.712,-1.7) node (p2) {} ;
       \path (0.115,3.2) node (p3) {} ;

\path (-1.07,-0.656) node (p4) {} ;
         \path (1.105,-0.606) node (p5) {} ;
       \path (-0.03,1.256) node (p6) {} ;

   \draw (p4)  --  (p5);
         \draw (p4)  --  (p6);
        \draw (p6)  --  (p5);
        \draw (p1)  --  (p2);
         \draw (p1)  --  (p3);
        \draw (p3)  --  (p2);
         \draw (p4)  --  (p1);
         \draw (p5)  --  (p2);
        \draw (p6)  --  (p3);
                
         \node [rectangle,draw=white, fill=white] (a) at (0,-3) {(d)};        
                  \end{tikzpicture}  
                 \end{center}
\caption{Generating a symmetrically averaged framework $(G,\x)$ with $\mathcal{C}_3$ symmetry from a non-symmetric framework $(G,\p)$: the framework in (b) is $(G,\q)$, where $\q=(\gamma \cdot \p)$; the framework in (c) is $(G,\q')$, where $\q'=(\gamma^2 \cdot \p)$. Finally, the framework $(G,\x)$ in (d) is the average $(G,A\p)$ of the ones in (a), (b) and (c) and clearly has $\mathcal{C}_3$ symmetry.}
\label{fig:avgexgen}
\end{figure}

\section{Properties of the averaging map}\label{sec:symavprop} 
Now we are ready to show that the map $A$ does, in fact, 
find the closest $(\Gamma,\tau)$-symmetric framework to $(G,\p)$, 
in a sense we now make precise.
We show that $A$ is the orthogonal projection onto the subspace of configurations 
that are $(\Gamma,\tau)$-symmetric.  To discuss orthogonality, we 
need an inner product on configurations of $n$ points.  We use 
the one induced by the Euclidean inner product on $\RR^d$:
\[
    \iprod{\p}{\q} = \sum_{i=1}^n \iprod{\p_i}{\q_i}
\]
That $A$ is an orthogonal projection follows from the facts that: $A$ is an idempotent 
linear map (this establishes that $A$ is a linear projection to its image) and, 
moreover, self-adjoint (which implies that its kernel and image are orthogonal 
complements).

Let us start with linearity.  This is quick from the 
linearity of the isometries $\tau(\gamma)$.  Indeed, if 
$\alpha\in \RR$, and configurations $\p$ and $\q$ are 
given, we have, for each $i$:
\[
    (A(\alpha \p + \q))_i = 
        \sum_{\gamma\in \Gamma} \tau(\gamma^{-1})(\alpha \p_{\gamma(i)} + \q_{\gamma(i)})
        = \left(\alpha \sum_{\gamma\in \Gamma} \tau(\gamma^{-1})\p_{\gamma(i)}\right) + 
        \left(\sum_{\gamma\in \Gamma} \tau(\gamma^{-1})\q_{\gamma(i)}\right) =
        \alpha A(\p)_i + A(\q)_i
\]
Now we show that $A$ is the identity on symmetric configurations.  Recall that 
if $\p$ is already $(\Gamma,\tau)$-symmetric, then $\gamma\cdot \p = \p$ for 
each $\gamma\in \Gamma$. So we get, for a $(\Gamma,\tau)$-symmetric $\p$:
\[
    A\p = \frac{1}{|\Gamma|}\sum_{\gamma\in \Gamma} \gamma\cdot \p
        = \frac{1}{|\Gamma|}\sum_{\gamma\in \Gamma} \p = \p
\]
This verification, in particular, implies that 
every symmetric configuration is in the image of $A$.

For the reverse inclusion, we need to show that $A\p$ is, for 
any $\p$, $(\Gamma,\tau)$-symmetric.  With the previous observation, this tells
us that $A$ is idempotent and that its image is exactly the 
$(\Gamma,\tau)$-symmetric configurations.  
To this end, let $\eta\in \Gamma$ and a 
vertex $i$ be given.  Now compute 
\[
    (A\p)_{\eta(i)} = 
    \frac{1}{|\Gamma|} \sum_{\gamma\in \Gamma} \tau(\gamma^{-1})\p_{(\gamma\eta)(i)}
    =  \frac{1}{|\Gamma|}
    \sum_{\gamma\eta^{-1}\in \Gamma}\tau(\eta\gamma^{-1})\p_{\gamma(i)}
    = \tau(\eta)\frac{1}{|\Gamma|} \sum_{\gamma\in \Gamma} \tau(\gamma^{-1})\p_{\gamma(i)} = \tau(\eta)(A\p)_i
\]
where we used the ``group action'' property $\gamma(\eta(i)) = (\gamma\eta)(i)$
and  that $\Gamma \eta = \Gamma$ to reindex the sum.  As $\eta$ and $i$ 
were arbitrary, we see that $A\p$ is $(\Gamma,\tau)$-symmetric (fulfilling a 
promise from the motivating discussion in the process).
Thus, $A$ is indeed a linear projection. 

The remaining step is to show that, for arbitrary configurations 
$\p$ and $\q$, 
$$\langle A\p,\q\rangle=\langle \p,A\q\rangle$$
Letting $\p$ and $\q$ be given, we compute 
\[
\begin{split}
    \iprod{A\p}{\q} & = \sum_{i=1}^n \iprod{(A\p)_i}{\q_i} \\
    & = \sum_{i=1}^n \iprod{\frac{1}{|\Gamma|}\sum_{\gamma\in \Gamma} \tau(\gamma^{-1})\p_{\gamma(i)}}{\q_i} \\
    & = \sum_{i=1}^n \frac{1}{|\Gamma|} \sum_{\gamma\in \Gamma} \iprod{\tau(\gamma^{-1})\p_{\gamma(i)}}{\q_i} \\ 
    & =\frac{1}{|\Gamma|}\sum_{i=1}^n \sum_{\gamma\in \Gamma} \iprod{\p_{\gamma(i)}}{\tau(\gamma)\q_i}
\end{split}
\]
where we have repeatedly used bilinearity of the inner product and, in the last step, that $\tau(\gamma)$
is orthogonal, so that 
\[
    \iprod{\tau(\gamma^{-1})\p_{\gamma(i)}}{\q_i} = \iprod{\tau(\gamma)\tau(\gamma^{-1})\p_{\gamma(i)}}{\tau(\gamma)\q_i}
    = \iprod{\p_{\gamma(i)}}{\tau(\gamma)\q_i}
\]
because $\tau(\gamma)\tau(\gamma^{-1})$ must be the identity.  Continuing our main computation, we 
reorder sums and reindex
\[
\begin{split}
    \frac{1}{|\Gamma|}\sum_{i=1}^n \sum_{\gamma\in \Gamma} \iprod{\p_{\gamma(i)}}{\tau(\gamma)\q_i} 
    & =   \frac{1}{|\Gamma|} \sum_{\gamma\in \Gamma} \sum_{i=1}^n \iprod{\p_{\gamma(i)}}{\tau(\gamma)\q_i} \\
    & = \frac{1}{|\Gamma|} \sum_{\gamma\in \Gamma} 
    \sum_{\gamma^{-1}(i)=1}^n \iprod{\p_{\gamma(\gamma^{-1}(i))}}{\tau(\gamma)\q_{\gamma^{-1}(i)}} \\
    & = \frac{1}{|\Gamma|} \sum_{\gamma\in \Gamma} 
    \sum_{\gamma^{-1}(i)=1}^n \iprod{\p_{i}}{\tau(\gamma)\q_{\gamma^{-1}(i)}} \\ 
    & =  \frac{1}{|\Gamma|} \sum_{i=1}^n \sum_{\gamma\in \Gamma} 
    \iprod{\p_{i}}{\tau(\gamma^{-1})\q_{\gamma(i)}} \\
    & = 
        \iprod{\p}{ \sum_{i=1}^n \frac{1}{|\Gamma|} \sum_{\gamma\in \Gamma} \tau(\gamma^{-1})\q_{\gamma(i)}} \\
    & = 
        \sum_{i=1}^n \iprod{\p_i}{(A\q)_i} \\
    & = \iprod{\p}{A\q}
\end{split}
\]
This  computation shows that $A$ is self-adjoint, so 
we have verified all the properties we needed to make it an orthogonal 
projection.

\subsection{Connections to representation-theoretic approaches}

We have given a direct derivation of the averaging procedure, since it has a geometric 
intuition from frameworks and doesn't require any theory.  That said, we can 
interpret what happens above in a very general setting.  If $G$ is a 
finite group and $\rho : G \to \operatorname{GL}(V)$ is any linear representation 
for some $F$-vector space $V$ and $|G|$ is invertible in $F$, then 
the averaging operator $A_\rho$ defined by
\[
    v\mapsto |G|^{-1}\sum_{g\in G} \rho(g)v
\]
is well-known to be a linear projection onto the invariant subspace 
\[
    V^G = \{ v : \rho(g)v = v \quad \text{for all $g\in G$}\}
\]
To apply the general theory, we observe that the  map 
\[
    \p \mapsto \sum_{i=1}^n \e_i \otimes \p_i,
\]
where $\e_i$ is the $i$th canonical basis vector, is a linear isomorphism between 
the space of $n$-point configurations in $\RR^d$ and the 
tensor product $\RR^n \otimes \RR^d$.
The inner product is defined on the basis vectors of 
$\RR^n\otimes \RR^d$ by
\[
    \iprod{\e_i\otimes \e_j}{\e_k\otimes \e_\ell} = 1 \qquad 
    \text{if and only if }\qquad i = k\quad \text{and}\quad j = \ell
\]
and linear extension.   Finally, 
from $\Gamma$ and $\tau$, we end up with a linear action $\rho : \Gamma\to \operatorname{GL}(\RR^d \otimes \RR^n)$
defined by 
\[
    \rho(\gamma)(\e_i\otimes \p_i)  =  \e_{\gamma(i)} \otimes \tau(\gamma)\p_i
\]
It's a quick check that this preserves the inner product, so averaging over this $\rho$ 
gives an $A_\rho$ which is an orthogonal projection, which is what we verified 
above.  With respect to the 
standard basis, the matrices of the operators $\rho(\gamma)$ are 
\[
    \operatorname{Mat}(\rho(\gamma)) =   V_\gamma \otimes \tau(\gamma) \in \RR^{dn\times dn}
\]
 where the ``$\otimes$'' is the Kronecker product, and $V_\gamma$ is the vertex permutation matrix 
\[
    V_\gamma = \begin{pmatrix}
        \e_{\gamma(1)} &  \e_{\gamma(2)} & \cdots & \e_{\gamma(n)}
    \end{pmatrix}.
\]
These are the ``external representation matrices'' introduced in \cite{kangwai2000,fow00}. Finally, we have 
\[
    \operatorname{Mat}(A) = \frac{1}{|\Gamma|}\sum_{\gamma\in \Gamma} V_\gamma\otimes \tau(\gamma),
\]
which is consistent with our previous derivation.

\section{Types of self-stress: local vs. extensive}\label{sec:stresses}

In this section we establish mathematically rigorous definitions of different types of self-stress that play an important role in practical engineering applications, such as gridshell design. In these applications, self-stresses that are localised, i.e., have non-zero stress coefficients only on some isolated part of the framework, are of significantly less interest  than self-stresses with more ``extensive support".

We will  propose a hierarchy of definitions for local and extensive self-stresses, starting with  notions of local and extensive that depend on the group $\Gamma < \Aut(G)$ (see Section~\ref{subsec:sstypes}). In Section~\ref{subsec:sseamp} we will then investigate the effect of the symmetric averaging map on these types of self-stresses. Finally, in Section~\ref{subsec:ssext} we will introduce a symmetry independent notion of extensiveness for self-stresses.

\subsection{$\Gamma$-localised and $\Gamma$-extensive self-stresses}\label{subsec:sstypes}

In the sequel, we fix a graph $G$ and $\Gamma < \Aut(G)$.  For the moment, $\tau$ plays less of a role. We propose the following natural notion of a ``$\Gamma$-localised" self-stress.

The \emph{(stress) support} of a self-stress $\mathbf \omega$ of a framework $(G,\p)$ is the set of edges $e$ of $G$ that have a non-zero stress-coefficient $\mathbf \omega_e$. Given a graph $G$ and $\Gamma<\textrm{Aut}(G)$,  the \emph{$\Gamma$-orbit} (or just \emph{orbit}) of a vertex $i$ (edge $e$) of $G$ is  $\Gamma i=\{\gamma(i)|\, \gamma\in \Gamma\}$ ($\Gamma e=\{\gamma(e)|\, \gamma\in \Gamma\}$, respectively). The \emph{$\Gamma$-orbit} (or simply \emph{orbit}) of a subgraph $H$ of $G$ is the subgraph $\Gamma H$ of $G$ consisting of the elements of the orbits of vertices and edges of $H$.

\begin{definition} (Strongly $\Gamma$-localised self-stress)
    Given a graph $G$ with $\Gamma < \Aut(G)$ and a 
 framework  $(G,\p)$, a self-stress is called \emph{strongly $\Gamma$-localised} if it is only supported on a subgraph $H$ of $G$ whose $\Gamma$-orbit consists of $|\Gamma|$ vertex-disjoint copies of  $H$.
\end{definition}
Informally, a strongly $\Gamma$-localised self-stress is supported on only one of the ``copies'' of $H$.  Different copies may have the 
same self-stress, possibly different self-stresses, or none at all, depending on the properties of $H$ and the 
induced frameworks on the copies.

With an eye to defining ``$\Gamma$-extensive", we describe a natural relaxation  of $\Gamma$-local that includes strongly $\Gamma$-local.
\begin{definition} (Weakly $\Gamma$-localised self-stress)
    Given a graph $G$ with $\Gamma < \Aut(G)$ and a 
framework  $(G,\p)$, a self-stress is called \emph{weakly $\Gamma$-localised} if for every edge orbit, the stress support does not contain all the edges of the orbit.
\end{definition}
We can see that if $\Gamma$ is non-trivial, then strongly $\Gamma$-localised implies weakly $\Gamma$-localised, since a strongly $\Gamma$-localised self-stress can contain only one edge per orbit.  (If the orbit is a single edge, then the $|\Gamma|$ copies of $H$ are not disjoint, provided that $\Gamma$ is non-trivial, which is why that condition was in the definition of strongly $\Gamma$-localised self-stress.)

The preceding definition lets us bootstrap a notion of a $\Gamma$-extensive self-stress, which is one that cannot be obtained 
by linear combinations of weakly $\Gamma$-localised self-stresses.
\begin{definition} ($\Gamma$-extensive self-stress)
    Given a graph $G$ with $\Gamma < \Aut(G)$ and a 
framework  $(G,\p)$, a self-stress is called \emph{$\Gamma$-extensive} if it is not in the linear span of its weakly $\Gamma$-localised self-stresses.
\end{definition}

It is important to keep in mind here that the notions of (weakly or strongly) $\Gamma$-localised and $\Gamma$-extensive self-stress depend on the group $\Gamma$. If the group $\Gamma$ gets larger, then we typically have fewer edge orbits and larger sizes of  orbits. So by increasing the size of the group, we can turn extensive to weakly or strongly localised self-stresses. (In the extreme but uninteresting case when $\Gamma$ is the trivial group, every non-zero self-stress is trivially $\Gamma$-extensive.)

In practical applications, such as gridshell design, one often deals with an initial form diagram (i.e. planar framework) which exhibits no symmetry, but whose configuration of points is ``close" to a symmetric configuration. This then suggests a $\Gamma$ (and $\tau$) for carrying out the self-stress analysis.

\subsection{Self-stress behaviour under the averaging map}\label{subsec:sseamp}

We will now investigate the changes in $\Gamma$-localised and $\Gamma$-extensive self-stresses  under the symmetric averaging map. By looking at some simple examples, we will see that,  in general, the averaging procedure can have a large variety of different effects on the self-stresses.

\medskip

\noindent \textbf{$\Gamma$-localised self-stresses.} The averaging map can both create and destroy $\Gamma$-localised self-stresses, as illustrated by the  examples in Figure~\ref{fig:1}. For these examples, the group $\Gamma$ is the group $\mathbb{Z}_2$ of order 2  and $\tau(\Gamma)$ is the reflection group $\mathcal{C}_s$ in the plane.

\begin{figure}[htp]
\begin{center}
\begin{tikzpicture}[very thick,scale=1]
   \tikzstyle{every node}=[circle, draw=black, fill=white, inner sep=0pt, minimum width=4pt];

   \path (-1,-0.5) node (p3)  {} ;
 \path (-1,0.5) node (p1)  {} ;
    \path (1,-0.5) node (p4)  {} ;
   \path (1,0.5) node (p2)  {} ;
\path (-2,0) node (l)  {} ;
   \path (1.1,0) node (r)  {} ;
\draw[thin, dashed] (0,-1)--(0,1);

\draw (p1)  --  (p2);
        \draw (p1)  --  (p3);
\draw (p4)  --  (p2);
        \draw (p4)  --  (p3);
\draw (l)  --  (p1);
        \draw (l)  --  (p3);
\draw (r)  --  (p2);
        \draw (r)  --  (p4);

 \node [draw=white, fill=white] (b) at (0,-1.3) {(a)};
\end{tikzpicture}
\hspace{0.2cm}
\begin{tikzpicture}[very thick,scale=1]
   \tikzstyle{every node}=[circle, draw=black, fill=white, inner sep=0pt, minimum width=4pt];

   \path (-1,-0.5) node (p3)  {} ;
 \path (-1,0.5) node (p1)  {} ;
    \path (1,-0.5) node (p4)  {} ;
   \path (1,0.5) node (p2)  {} ;
  \path (-1.5,0) node (l)  {} ;
   \path (1.5,0) node (r)  {} ;

\draw (p1)  --  (p2);
        \draw (p1)  --  (p3);
\draw (p4)  --  (p2);
        \draw (p4)  --  (p3);
\draw (l)  --  (p1);
        \draw (l)  --  (p3);
\draw (r)  --  (p2);
        \draw (r)  --  (p4);

\draw[thin, dashed] (0,-1)--(0,1);
   
 \node [draw=white, fill=white] (b) at (0,-1.3) {(b)};
\end{tikzpicture}
\hspace{1cm}
\begin{tikzpicture}[very thick,scale=1]
   \tikzstyle{every node}=[circle, draw=black, fill=white, inner sep=0pt, minimum width=4pt];

   \path (-1,-0.5) node (p3)  {} ;
 \path (-1,0.5) node (p1)  {} ;
    \path (1,-0.5) node (p4)  {} ;
   \path (1,0.5) node (p2)  {} ;
\path (-1.5,0) node (l)  {} ;
   \path (0.5,0) node (r)  {} ;

\draw (p1)  --  (p2);
        \draw (p1)  --  (p3);
\draw (p4)  --  (p2);
        \draw (p4)  --  (p3);
\draw (l)  --  (p1);
        \draw (l)  --  (p3);
\draw (r)  --  (p2);
        \draw (r)  --  (p4);

\draw[thin, dashed] (0,-1)--(0,1);
   
 \node [draw=white, fill=white] (b) at (0,-1.3) {(c)};
\end{tikzpicture}
\hspace{0.2cm}
\begin{tikzpicture}[very thick,scale=1]
   \tikzstyle{every node}=[circle, draw=black, fill=white, inner sep=0pt, minimum width=4pt];

   \path (-1,-0.5) node (p3)  {} ;
 \path (-1,0.5) node (p1)  {} ;
    \path (1,-0.5) node (p4)  {} ;
   \path (1,0.5) node (p2)  {} ;
  \path (-1.1,0) node (l)  {} ;
   \path (1.1,0) node (r)  {} ;

\draw (p1)  --  (p2);
        \draw (p1)  --  (p3);
\draw (p4)  --  (p2);
        \draw (p4)  --  (p3);
\draw (l)  --  (p1);
        \draw (l)  --  (p3);
\draw (r)  --  (p2);
        \draw (r)  --  (p4);
\draw[thin, dashed] (0,-1)--(0,1);
   
 \node [draw=white, fill=white] (b) at (0,-1.3) {(d)};
\end{tikzpicture}
\end{center}
\vspace{-0.6cm} \caption{$\mathbb{Z}_2$-localised self-stresses may appear or disappear under symmetric averaging. (a) A framework in the plane with one strongly $\mathbb{Z}_2$-localised self-stress due to the collinear triangle on the right (the centre vertex of the triangle is drawn slightly offset for visibility). The symmetrically averaged framework with reflection symmetry in (b) has no self-stress.  (c) A framework in the plane with no self-stresses whose symmetrically averaged framework under reflection symmetry in (d) has two strongly $\mathbb{Z}_2$-localised self-stresses due to the two collinear triangles.} \label{fig:1}
\end{figure}
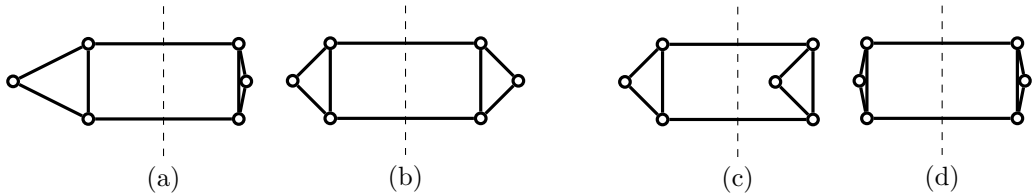

Similar and less degenerate examples can easily be constructed by replacing the two triangles in Figure~\ref{fig:1} by prism graphs (such as the planar cube graph $Q_3$, for example, whose  self-stressed configurations in the plane are well known \cite{wwstresses,NSW21}).

The following result shows that ``generically'' $\Gamma$-symmetric averaging cancels all strongly $\Gamma$-localised self-stresses.

A $(\Gamma,\tau)$-symmetric framework $(G,\p)$ is called \emph{$(\Gamma,\tau)$-generic} if $\textrm{rank}R(\p)=\textrm{max}\{\textrm{rank} R(\p')|\, \p'\in \mathbb{R}^{d|V|} \textrm{ and } (G,\p') \textrm{ is }(\Gamma,\tau)\textrm{-symmetric}\}$. It is easy to see that ``almost all" realisations of $G$ that are  $(\Gamma,\tau)$-symmetric are $(\Gamma,\tau)$-generic (in the sense that the $(\Gamma,\tau)$-generic frameworks are open and dense in the set of $(\Gamma,\tau)$-symmetric frameworks).

\begin{theorem}
    Let $\Gamma$ be a non-trivial group and let $G$ be a  graph whose $(\Gamma,\tau)$-generic realisations as a framework have no non-trivial self-stress. Suppose the (not necessarily symmetric) framework $(G,\p)$ has a strongly $\Gamma$-localised self-stress $\mathbf \omega$ which is supported on a subgraph $H$ of $G$. Then for almost all positions of the vertices of a copy of $H$ from the $\Gamma$-orbit of $H$, the averaging map for $\Gamma$ and $\tau$ will yield a $(\Gamma,\tau)$-symmetric framework that no longer has any non-trivial self-stress. 
\end{theorem}
\begin{proof}
     Let $H'$ be an element from the $\Gamma$-orbit of $H$. When we perturb the positions $\p|_{V(H')}$ of the vertices of $H'$ within an open neighbourhood $N$, then the symmetric averaging map $A$ for $\Gamma$ and $\tau$ yields a corresponding open neighbourhood $N'$ of $(\Gamma,\tau)$-symmetric configurations of the vertices of $\Gamma H$, which includes the $(\Gamma,\tau)$-symmetric configuration  $A\p|_{V(\Gamma H)}$. Any perturbation of the vertices outside $V(H')$ does not effect the self-stress properties of a realisation of $H'$.  Thus, since $(\Gamma,\tau)$-generic realisations of $G$ have no non-trivial self-stress, 
        almost all configurations $\q$ of $V(H')$ in $N$ have the property that  $A$ applied to $\q$ and $\p|_{V\setminus V(H')}$ yield frameworks with no non-trivial self-stress. 
\end{proof}

It is natural to ask how $\Gamma$-localised self-stresses typically appear in realisations of  independent and \emph{non-rigid} graphs with non-degenerate configurations (say configurations in general position). This seems to be worthy of further investigation in the future.

\medskip

\noindent \textbf{$\Gamma$-extensive self-stresses.} 
The averaging map can also  create, preserve, or destroy $\Gamma$-extensive self-stresses.

We again provide examples for the case when the group $\Gamma$ is the group $\mathbb{Z}_2$ of order 2  and $\tau(\Gamma)$ is the reflection group $\mathcal{C}_s$ in the plane. Figure~\ref{fig:prism}(a,b) shows that a $\Gamma$-extensive self-stress may be gained via symmetric averaging. Figure~\ref{fig:prism}(c,d) shows that a $\Gamma$-extensive self-stress may simply be preserved under the averaging map. 

\begin{figure}[htp]
\begin{center}
   \begin{tikzpicture}[very thick,scale=0.5]
\tikzstyle{every node}=[circle, draw=black, fill=white, inner sep=0pt, minimum width=4pt];
        \path (-0.8,-0.1) node (p1) {} ;
        \path (1.2,0.25) node (p2) {} ;
        \path (-0.25,1.63333) node (p3) {} ;
                
         \path (-3,-0.66666) node (p11) {} ;
        \path (3,-0.66666) node (p22) {} ;
        \path (0,3.333333) node (p33) {} ;
           
        \draw (p1)  --  (p2);
         \draw (p1)  --  (p3);
        \draw (p3)  --  (p2);
        \draw (p11)  --  (p22);
         \draw (p11)  --  (p33);
        \draw (p33)  --  (p22);
         \draw (p1)  --  (p11);
         \draw (p2)  --  (p22);
        \draw (p3)  --  (p33);
\draw[thin, dashed] (0,-1.3)--(0,4);
                
         \node [rectangle,draw=white, fill=white] (a) at (0,-1.7) {(a)};        
                  \end{tikzpicture}
                  \hspace{0.2cm}
                   \begin{tikzpicture}[very thick,scale=0.5]
\tikzstyle{every node}=[circle, draw=black, fill=white, inner sep=0pt, minimum width=4pt];
        \path (-1,0) node (p1) {} ;
        \path (1,0) node (p2) {} ;
        \path (0,1.33333) node (p3) {} ;
        
           \node [draw=white, fill=white] (a) at (-0.5,1.7) {};
          \node [draw=white, fill=white] (a) at (-1,-1.1) {};  
         \node [draw=white, fill=white] (a) at (1.7,0.2) {};
               
         \path (-3,-0.66666) node (p11) {} ;
        \path (3,-0.66666) node (p22) {} ;
        \path (0,3.333333) node (p33) {} ;

        \draw (p1)  --  (p2);
         \draw (p1)  --  (p3);
        \draw (p3)  --  (p2);
        \draw (p11)  --  (p22);
         \draw (p11)  --  (p33);
        \draw (p33)  --  (p22);
         \draw (p1)  --  (p11);
         \draw (p2)  --  (p22);
        \draw (p3)  --  (p33);
        \draw[thin, dashed] (0,-1.3)--(0,4);
             
        
            \node [rectangle,draw=white, fill=white] (a) at (0,-1.7) {(b)};
      \end{tikzpicture}
      \hspace{1cm}
\begin{tikzpicture}[very thick,scale=0.8]
   \tikzstyle{every node}=[circle, draw=black, fill=white, inner sep=0pt, minimum width=4pt];

   \path (-2,0) node (p1)  {} ;
 \path (-0.5,0) node (p2)  {} ;
       \path (0,-1) node (p3)  {} ;
 \path (0,1) node (p4)  {} ;
 \path (2,0) node (p5)  {} ;
 \path (1.5,0) node (p6)  {} ;
 
\draw (p1)  --  (p2);
  \draw (p1)  --  (p3);
\draw (p1)  --  (p4);
  \draw (p2)  --  (p3);
\draw (p2)  --  (p4);

\draw (p6)  --  (p5);
  \draw (p6)  --  (p3);
\draw (p6)  --  (p4);
  \draw (p5)  --  (p3);
\draw (p5)  --  (p4);

\draw[thin, dashed] (0,-1.3)--(0,1.3);
   
 \node [draw=white, fill=white] (b) at (0,-1.6) {(c)};
 \end{tikzpicture}
 \hspace{0.2cm}
\begin{tikzpicture}[very thick,scale=0.8]
   \tikzstyle{every node}=[circle, draw=black, fill=white, inner sep=0pt, minimum width=4pt];

   \path (-2,0) node (p1)  {} ;
 \path (-1,0) node (p2)  {} ;
       \path (0,-1) node (p3)  {} ;
 \path (0,1) node (p4)  {} ;
 \path (2,0) node (p5)  {} ;
 \path (1,0) node (p6)  {} ;
 
\draw (p1)  --  (p2);
  \draw (p1)  --  (p3);
\draw (p1)  --  (p4);
  \draw (p2)  --  (p3);
\draw (p2)  --  (p4);

\draw (p6)  --  (p5);
  \draw (p6)  --  (p3);
\draw (p6)  --  (p4);
  \draw (p5)  --  (p3);
\draw (p5)  --  (p4);

\draw[thin, dashed] (0,-1.3)--(0,1.3);
   
 \node [draw=white, fill=white] (b) at (0,-1.6) {(d)};
\end{tikzpicture}
                \end{center}
\caption{(a) A framework in the plane with no self-stress whose symmetrically averaged framework under reflection symmetry (b) has a $\mathbb{Z}_2$-extensive self-stress, which is fully-symmetric. The symmetrically averaged framework in (d) has only a single self-stress, which is  $\mathbb{Z}_2$-extensive, just like the original non-symmetric framework in (c). The self-stress of (d) is anti-symmetric with respect to the reflection symmetry (i.e. the reflection reverses the sign of the stress coefficients).}
\label{fig:prism}
\end{figure}
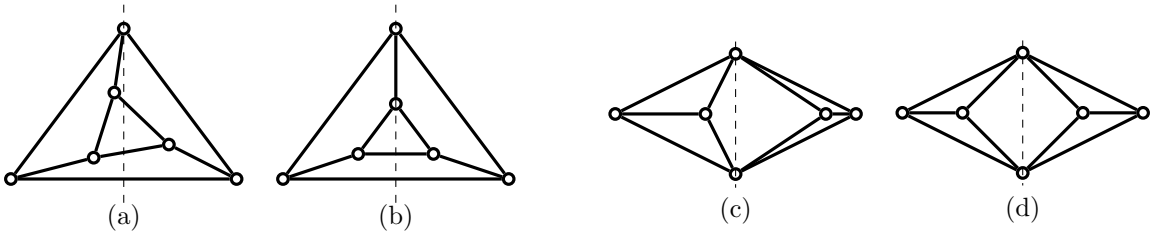

Note that the $\Gamma$-extensive self-stresses in the  averaged frameworks may be ``fully-symmetric", in the sense that the stress coefficients on the edges of the same orbit are the same (as is the case in the well-known example of Figure~\ref{fig:prism}(b), which shows the triangular prism graph in a Desargues configuration), or ``anti-symmetric", in the sense that the stress coefficients on the two edges of the same orbit are related by a  sign change (as is the case in Figure~\ref{fig:prism}(d)).
Formally,  fully- and anti-symmetric self-stresses lie in  the  subspaces of $\mathbb{R}^{|E|}$ which are the invariant subspaces of  the edge permutation representation of the graph under $\mathbb{Z}_2$ corresponding to the trivial and non-trivial irreducible representation of $\mathbb{Z}_2$, respectively. See for example \cite{kangwai1999b,kangwai2000,fow00,sch10a} for details. 

Both fully-symmetric and anti-symmetric self-stresses (and  generalisations of these symmetry types for higher order groups) are useful in engineering applications such as gridshell design \cite{smmb,mmsb,schmil}.

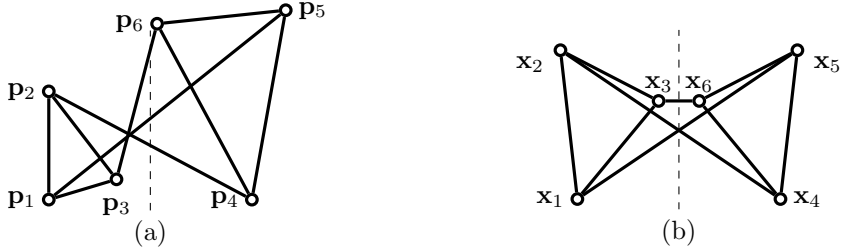
\begin{figure}[htp]
\begin{center}
\begin{tikzpicture}[very thick,scale=0.9]
   \tikzstyle{every node}=[circle, draw=black, fill=white, inner sep=0pt, minimum width=4pt];

   \path (-1.5,-0.8) node (p3)  {} ;
 \path (-1.5,0.8) node (p1)  {} ;
    \path (1.5,-0.8) node (p4)  {} ;
   \path (2,2) node (p2)  {} ;
\path (-0.5,-0.5) node (l)  {} ;
   \path (0.1,1.8) node (r)  {} ;

\node [rectangle,draw=white, fill=white] (a) at (-1.9,-0.8) {$\p_1$}; 
\node [rectangle,draw=white, fill=white] (a) at (-1.9,0.8) {$\p_2$};
\node [rectangle,draw=white, fill=white] (a) at (-0.5,-0.9) {$\p_3$};
\node [rectangle,draw=white, fill=white] (a) at (-0.3,1.8) {$\p_6$};
\node [rectangle,draw=white, fill=white] (a) at (1.1,-0.8) {$\p_4$};
\node [rectangle,draw=white, fill=white] (a) at (2.4,2) {$\p_5$};

\draw (p1)  --  (p4);
 \draw (p2)  --  (p3);
 \draw (l)  --  (r);

        \draw (p1)  --  (p3);
\draw (p4)  --  (p2);
       \draw (l)  --  (p1);
        \draw (l)  --  (p3);
\draw (r)  --  (p2);
        \draw (r)  --  (p4);

\draw[thin, dashed] (0,-1.2)--(0,1.7);
   
 \node [draw=white, fill=white] (b) at (0,-1.3) {(a)};
\end{tikzpicture}
\hspace{1cm}
\begin{tikzpicture}[very thick,scale=1]
   \tikzstyle{every node}=[circle, draw=black, fill=white, inner sep=0pt, minimum width=4pt];

   
\end{tikzpicture}
\hspace{1cm}
\begin{tikzpicture}[very thick,scale=0.9]
   \tikzstyle{every node}=[circle, draw=black, fill=white, inner sep=0pt, minimum width=4pt];

   \path (-1.5,-0.8) node (p3)  {} ;
 \path (-1.75,1.4) node (p1)  {} ;
    \path (1.5,-0.8) node (p4)  {} ;
   \path (1.75,1.4) node (p2)  {} ;
\path (-0.3,0.65) node (l)  {} ;
   \path (0.3,0.65) node (r)  {} ;

\node [rectangle,draw=white, fill=white] (a) at (-1.9,-0.8) {$\x_1$}; 
\node [rectangle,draw=white, fill=white] (a) at (-2.2,1.2) {$\x_2$};
\node [rectangle,draw=white, fill=white] (a) at (-0.3,0.9) {$\x_3$};
\node [rectangle,draw=white, fill=white] (a) at (0.3,0.9) {$\x_6$};
\node [rectangle,draw=white, fill=white] (a) at (1.9,-0.8) {$\x_4$};
\node [rectangle,draw=white, fill=white] (a) at (2.2,1.2) {$\x_5$};

\draw (p1)  --  (p4);
 \draw (p2)  --  (p3);
 \draw (l)  --  (r);

        \draw (p1)  --  (p3);
\draw (p4)  --  (p2);
       \draw (l)  --  (p1);
        \draw (l)  --  (p3);
\draw (r)  --  (p2);
        \draw (r)  --  (p4);

\draw[thin, dashed] (0,-1.2)--(0,1.7);
   
 \node [draw=white, fill=white] (b) at (0,-1.3) {(b)};
\end{tikzpicture}
\hspace{1cm}
\begin{tikzpicture}[very thick,scale=1]
   \tikzstyle{every node}=[circle, draw=black, fill=white, inner sep=0pt, minimum width=4pt];

   
\end{tikzpicture}
\end{center}
\vspace{-0.6cm} \caption{A realisation $(G,\p)$ of the triangular prism graph in the plane with a $\mathbb{Z}_2$-extensive self-stress, where $\mathbb{Z}_2$ is generated by the automorphism $(1,4)(2,5)(3,6)$ of $G$ (a). After symmetric averaging with the group $\tau(\mathbb{Z}_2)=\mathcal{C}_s$, the three bars connecting the triangles no longer meet in a point, and hence the averaged framework $(G,\x)$ in (b) is independent.} \label{fig:extloss}
\end{figure}

Figure~\ref{fig:extloss} shows that $\Gamma$-extensive self-stresses may also be lost under symmetric averaging. Note that the configuration in Figure~\ref{fig:extloss}(a) is not close to being mirror-symmetric (in particular, the $y$-coordinates of $\p_3$ and $\p_6$ are far away from each other). However, the self-stress may also be lost for configurations that are close to being mirror-symmetric, as we see as follows.

Choose $\p_1,\p_2,\p_4,\p_5$ so that $\p_4$ and $\p_5$ are the respective mirror-symmetric copies of $\p_1$ and $\p_2$, and the edges $\p_1\p_5$ and $\p_2\p_4$ meet on the mirror line at the origin. Further, for $r>0$ choose an arbitrarily small $\epsilon>0$ so that with $\p_3=(-\epsilon,-r)$ and $\p_6=(2\epsilon, 2r)$, the edge $\p_3\p_6$ also goes through the origin. So the framework $(G,\p)$ does not quite have reflection symmetry, but has a ($\mathbb{Z}_2$-extensive) self-stress. Since  $\frac{\epsilon+2\epsilon}{2}=\frac{3\epsilon}{2}$ and $\frac{-r+2r}{2}=\frac{r}{2}$, in the symmetrically averaged framework $(G,\x)$, the points $\x_3$ and $\x_6$ will have coordinates $(-\frac{3\epsilon}{2},\frac{r}{2})$ and $(\frac{3\epsilon}{2},\frac{r}{2})$, so that the edge $\p_3\p_6$ no longer goes through the origin and the self-stress is lost. We may now choose $r$ arbitrarily close to zero so that the original framework comes arbitarily close to being mirror-symmetric.

We note that if we had averaged the framework in Figure~\ref{fig:extloss}(a) with the group $\tau(\mathbb{Z}_2)=\mathcal{C}_s$, where $\mathbb{Z}_2$ is generated by the automorphism $(1,5)(2,4)(3,6)$ of $G$, then the averaged framework would have three bars that are perpendicular to the mirror line, resulting in a Desargues configuration (as in Figure~\ref{fig:desarguestypes}(a)) that still has a $\mathbb{Z}_2$-extensive self-stress. 

So the choice of the group $\Gamma$  is crucial when applying the symmetric averaging procedure. Further, for a fixed $\Gamma$, different choices of the faithful representation $\tau$ of $\Gamma$ can also result in frameworks with different self-stress properties (although there exist group pairings, such as $\mathcal{C}_s$ and $\mathcal{C}_2$ in the plane,  where under suitable conditions these properties are equivalent for a given fixed $\Gamma<\textrm{Aut}(G)$, as shown in \cite{cnsw}).

In Section~\ref{sec:alg} we will present an algorithm which, for a given graph $G$, yields the subgroup(s) $\Gamma$ of $\Aut(G)$ and the corresponding  representations $\tau$ of $\Gamma$  for which the symmetry-extended Maxwell counting rule \cite{fow00} detects the maximum number of self-stresses in non-degenerate $(\Gamma,\tau)$-generic  realisations of $G$. These pairs $(\Gamma,\tau)$ provided by this algorithm are natural choices for the application of the  symmetric averaging procedure when trying to maximise states of self-stress.

\medskip

\noindent \textbf{$\Gamma$-localised versus $\Gamma$-extensive self-stresses.} We now show that $\Gamma$-localised self-stresses may become $\Gamma$-extensive under symmetric averaging and vice versa.

First, Figure~\ref{fig:locext1} shows that the averaging map may turn a $\Gamma$-localised self-stress into a $\Gamma$-extensive self-stress.  

\begin{figure}[htp]
\begin{center}
\begin{tikzpicture}[very thick,scale=1]
   \tikzstyle{every node}=[circle, draw=black, fill=white, inner sep=0pt, minimum width=4pt];

   \path (-1,-0.5) node (p3)  {} ;
 \path (-1,0.5) node (p1)  {} ;
    \path (1,-0.5) node (p4)  {} ;
   \path (1,0.5) node (p2)  {} ;
\path (-2,0) node (l)  {} ;
   \path (1.1,0) node (r)  {} ;
\draw[thin, dashed] (0,-1)--(0,1);

\draw (p1)  --  (p2);
        \draw (p1)  --  (p3);
\draw (p4)  --  (p2);
        \draw (p4)  --  (p3);
\draw (l)  --  (p1);
        \draw (l)  --  (p3);
\draw (r)  --  (p2);
        \draw (r)  --  (p4);
 \draw (r)  --  (l);
   
 \node [draw=white, fill=white] (b) at (0,-1.3) {(a)};
\end{tikzpicture}
\hspace{1cm}
\begin{tikzpicture}[very thick,scale=1]
   \tikzstyle{every node}=[circle, draw=black, fill=white, inner sep=0pt, minimum width=4pt];

   \path (-1,-0.5) node (p3)  {} ;
 \path (-1,0.5) node (p1)  {} ;
    \path (1,-0.5) node (p4)  {} ;
   \path (1,0.5) node (p2)  {} ;
  \path (-1.5,0) node (l)  {} ;
   \path (1.5,0) node (r)  {} ;

\draw (p1)  --  (p2);
        \draw (p1)  --  (p3);
\draw (p4)  --  (p2);
        \draw (p4)  --  (p3);
\draw (l)  --  (p1);
        \draw (l)  --  (p3);
\draw (r)  --  (p2);
        \draw (r)  --  (p4);
 \draw (r)  --  (l);
\draw[thin, dashed] (0,-1)--(0,1);
   
 \node [draw=white, fill=white] (b) at (0,-1.3) {(b)};
\end{tikzpicture}
\end{center}
\vspace{-0.6cm} \caption{The $\mathbb{Z}_2$-localised self-stress of the framework in (a), which is supported on the collapsed triangle, is turned into a $\mathbb{Z}_2$-extensive self-stress (which is non-zero on every edge) under symmetric averaging with the reflection group (b).} \label{fig:locext1}
\end{figure}
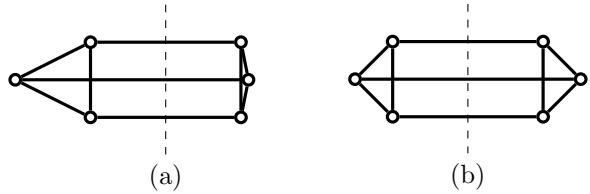

Other less degenerate versions of this example for reflection symmetry can easily be constructed. Consider, for example,  two vertex-disjoint triangular prism graphs, one on each side of the mirror, that are connected by three vertex-disjoint nearly parallel edges, and one of which is realised with a strongly $\Gamma$-localised self-stress. Then symmetric averaging with the reflection symmetry can destroy the $\Gamma$-localised self-stress and create a $\Gamma$-extensive self-stress which is non-zero on the three connecting edges, as they have become parallel.

Symmetric averaging may also turn $\Gamma$-extensive self-stresses into $\Gamma$-localised self-stresses, as shown in Figure~\ref{fig:locext2}. Again, less degenerate examples are easily constructed.

\begin{figure}[htp]
\begin{center}
\begin{tikzpicture}[very thick,scale=1]
   \tikzstyle{every node}=[circle, draw=black, fill=white, inner sep=0pt, minimum width=4pt];

   \path (-1,-0.5) node (p3)  {} ;
 \path (-1,0.5) node (p1)  {} ;
    \path (1,-0.5) node (p4)  {} ;
   \path (1,0.5) node (p2)  {} ;
\path (-1.5,0) node (l)  {} ;
   \path (0.5,0) node (r)  {} ;

\draw (p1)  --  (p2);
        \draw (p1)  --  (p3);
\draw (p4)  --  (p2);
        \draw (p4)  --  (p3);
\draw (l)  --  (p1);
        \draw (l)  --  (p3);
\draw (r)  --  (p2);
        \draw (r)  --  (p4);
 \draw (r)  --  (l);
\draw[thin, dashed] (0,-1)--(0,1);
   
 \node [draw=white, fill=white] (b) at (0,-1.3) {(a)};
\end{tikzpicture}
\hspace{1cm}
\begin{tikzpicture}[very thick,scale=1]
   \tikzstyle{every node}=[circle, draw=black, fill=white, inner sep=0pt, minimum width=4pt];

   \path (-1,-0.5) node (p3)  {} ;
 \path (-1,0.5) node (p1)  {} ;
    \path (1,-0.5) node (p4)  {} ;
   \path (1,0.5) node (p2)  {} ;
  \path (-1.1,0) node (l)  {} ;
   \path (1.1,0) node (r)  {} ;

\draw (p1)  --  (p2);
        \draw (p1)  --  (p3);
\draw (p4)  --  (p2);
        \draw (p4)  --  (p3);
\draw (l)  --  (p1);
        \draw (l)  --  (p3);
\draw (r)  --  (p2);
        \draw (r)  --  (p4);
         \draw (r)  --  (l);
\draw[thin, dashed] (0,-1)--(0,1);
   
 \node [draw=white, fill=white] (b) at (0,-1.3) {(b)};
\end{tikzpicture}
\end{center}
\vspace{-0.6cm} \caption{The framework in (a) has a $\mathbb{Z}_2$-extensive self-stress which is non-zero on every edge. Under symmetric averaging with the reflection group, this self-stress disappears and is replaced with two strongly $\mathbb{Z}_2$-localised self-stresses, one for each of the collapsed triangles.} \label{fig:locext2}
\end{figure}
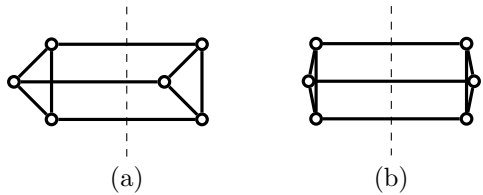

This also illustrates that  degenerating  a configuration further and further in the attempt to create additional self-stresses may be detrimental. In a framework on a triangular prism with a $\Gamma$-extensive self-stress that is non-zero on every edge, for example, collapsing a triangle is a bad idea, as it destroys the $\Gamma$-extensive property of the self-stress, as seen in the examples above. This shows the subtlety of the problem of creating $\Gamma$-extensive self-stresses.

\subsection{General localised and extensive self-stresses}\label{subsec:ssext}

In gridshell applications, members with zero coefficients in a self-stress are usually less interesting, as they  produce a ``flat" dihedral angle of $\pi$ for the two faces incident to the edge in the corresponding polyhedral surface obtained from the Maxwell-Cremona lifting.  Hence, we are often interested in self-stresses that are ``generally extensive'' in the 
sense that all the coefficients are non-zero; i.e., they have full support.
The notion of a $\Gamma$-extensive self-stress, while useful for many examples of symmetric frameworks, 
is more restrictive,  because for certain classes of frameworks, the support of a 
$\Gamma$-extensive self-stress can still be rather small, even if $\Gamma$ is large.

\begin{figure}[htp]
\begin{center}
             \begin{tikzpicture}[very thick,scale=1]
\tikzstyle{every node}=[circle, draw=black, fill=white, inner sep=0pt, minimum width=4pt];
         \path (90:1.7cm) node (o1) {} ;
       \path (210:1.7cm) node (o2) {} ;
           \path (330:1.7cm) node (o3) {} ;
        
        \path (90:1cm) node (p1) {} ;
       \path (210:1cm) node (p2) {} ;
           \path (330:1cm) node (p3) {} ;

 \path (90:0.5cm) node (p11) {} ;
       \path (210:0.5cm) node (p22) {} ;
           \path (330:0.5cm) node (p33) {} ;

            \draw (o1)  --  (o2);
         \draw (o1)  --  (o3);
        \draw (o3)  --  (o2);
        \draw (p1)  --  (p2);
         \draw (p1)  --  (p3);
        \draw (p3)  --  (p2);
        \draw (p11)  --  (p22);
         \draw (p11)  --  (p33);
        \draw (p33)  --  (p22);
         \draw (p1)  --  (p11);
         \draw (p2)  --  (p22);
        \draw (p3)  --  (p33);
         \draw (o1)  --  (p1);
         \draw (o2)  --  (p2);
        \draw (o3)  --  (p3);

\draw[thin, dashed] (0,-1)--(0,1);
\draw[thin, dashed] (210:1cm)--(30:1cm);
     \draw[thin, dashed] (150:1cm)--(330:1cm);               
                  \end{tikzpicture}
       
\end{center}
\vspace{-0.6cm} \caption{A framework consisting of nested triangular prisms.} \label{fig:nestedprism}
\end{figure}
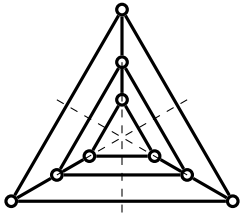

Consider, for example the framework with $\mathcal{C}_{3v}$ symmetry shown in Figure~\ref{fig:nestedprism}. This framework has a $2$-dimensional space of $\mathcal{C}_{3v}$-extensive self-stresses, one for the ``inner" triangular prism and one for the ``outer" one (each one being a Desargues configuration). Each of these self-stresses is only supported on 9 of the 15 edges. Such a ``nested'' framework is not suitable as a form diagram of a gridshell, as in any vertical Maxwell-Cremona lifting of  the form diagram corresponding to such a self-stress, the part of the framework that is unsupported by the self-stress would stay in the original plane, which is not acceptable from an architectural point of view. 

In the following, we will introduce a definition of  an extensive self-stress that does not depend on any symmetry of the structure, and is often better suited for engineering applications.
In the non-symmetric setting, there is no natural analogue of a strongly or weakly localised self-stress. However, a naive way to define extensive would be to simply say that the self-stress has full support. 
The reason this definition is also not up to the task is that, by taking 
linear combinations, we can change the supports of self-stresses.  Here are two extreme examples:
\begin{itemize}
    \item For a framework $(K_n,\p)$, where $K_n$ is the complete graph on $n$ vertices, every self-stress is in the span of self-stresses supported on $K_{d+2}$
        subgraphs.  No matter what the support of a self-stress is, it can be decomposed into local parts.
    \item For a framework $(C_n,\p)$, where $C_n$ is the cycle graph on $n$ vertices, and all the points are collinear, there is no self-stress supported on a proper subgraph, unless there is a 
        zero-length edge, so this self-stress has to be extensive in any reasonable definition.
\end{itemize}

A less trivial example is the nested prism framework in Figure~\ref{fig:nestedprism}. It has a self-stress with full support, but this self-stress is a linear combination  of the self-stresses that are only supported on the inner and outer prisms. Such a self-stress should not be considered ``extensive".

Instead we propose the following definition.

\begin{definition} (Extensive self-stress)
    Given a  framework  $(G,\p)$, a self-stress is called \emph{extensive} if it has full support and it is not the linear combination of self-stresses with strictly smaller support.
\end{definition}

A key advantage of this definition of extensiveness is that it is linear-algebraic in flavour, whereas ``full support" is not.  As the examples illustrate, 
one can use linear combinations to change supports without changing the underlying self-stress space. By formulating extensiveness in terms of  linear spans, 
we get what seems to be a more useful definition.

The example of the cycle graph realised as a framework on a line with no coincident vertices shows that extensive self-stresses exist.  Here is a more general 
construction.  Let $G$ be a graph so that, for any edge $ij$, $G-ij$ can be peeled to an empty graph by 
repeatedly removing degree $\le d$ vertices and their incident edges.  If $(G,\p)$ is a general position framework in 
$\RR^d$, any self-stress is extensive, because such a framework cannot support any self-stress that isn't non-zero on every edge.
To see this, suppose that some edge $ij$ is unstressed; the  graph resulting from removing $ij$ has a vertex of degree $\le d$, which cannot
be in equilibrium unless all its incident edges are unstressed.  Removing these exposes a new low-degree vertex, and, 
by induction, every edge is unstressed.  The interesting theoretical question is when there are extensive self-stresses in frameworks 
that do not arise this way.

Note that the definition of extensive is still somewhat restrictive. Returning to the complete graph example, since every self-stress 
in any framework $(K_n,\p)$ can be obtained from a linear combination of self-stresses supported on a $K_{d+2}$ subgraph 
or less, no self-stress in such a framework is extensive, unless $n\le d+2$. In fact, we have the following result.

\begin{prop}\label{prop:ext}
    If a framework $(G,\p)$ has an extensive self-stress, then this must be its only 
    self-stress (up to scale). 
\end{prop}
\begin{proof} Suppose $(G,\p)$ has an extensive self-stress $\mathbf\omega_1$ and another independent self-stress $\mathbf\omega_2$. By scaling we may assume that the stress coefficient on an edge $e$ of $G$ is equal to $1$ for both $\mathbf\omega_1$ and $\mathbf\omega_2$. If $\mathbf\omega_2$ does not have full support, then $\omega_3=\omega_2-\omega_1$ is a non-zero self-stress which is $0$ on $e$, and since $\omega_1=\omega_2-\omega_3$, we have a contradiction to the extensiveness of $\mathbf\omega_1$. If $\omega_2$ does have full support, then, by independence of $\mathbf\omega_1$ and $\mathbf\omega_2$, 
there must exist another edge $f\neq e$ for which $\mathbf\omega_1$ and $\mathbf\omega_2$ have distinct non-zero stress coefficients, say $(\mathbf\omega_2)_f=\lambda(\mathbf\omega_1)_f$ for $\lambda\neq 0,1$. Now, the self-stress $\mathbf\omega_3=\mathbf\omega_2- \mathbf\omega_1$   is zero on $e$ and non-zero on $f$, and the self-stress $\mathbf \omega_4= \mathbf\omega_2 - \lambda \mathbf\omega_1$  is zero on $f$ and non-zero on $e$. But $ \mathbf \omega_3 -\mathbf \omega_4= (\lambda-1) \mathbf\omega_1$ and so again we have a contradiction to the extensiveness of $\mathbf\omega_1$. 
\end{proof}

 Thus, we are mainly interested in sparse graphs, i.e., graphs $G$ satisfying the $d$-dimensional Maxwell sparsity count $m' \leq dn' - \frac{d(d+1)}{2}$ for all non-trivial subgraphs of $G$ with $n'$ vertices and $m'$ edges.

We note that while a framework with an extensive self-stress cannot have further self-stresses, by Proposition~\ref{prop:ext},  a planar $2$-dimensional framework with an extensive self-stress makes an ideal candidate for an initial form diagram to design a gridshell roof.
Having an extensive self-stress ensures a well-balanced initial structure, and additional self-stresses can be introduced by strategically adding edges to resolve the desired loads of the gridshell while maintaining structural efficiency. Moreover, identifying graphs that admit extensive self-stress realisations is valuable, as they help designers avoid localised self-stresses caused by over-counted or nested subgraph structures like the one in Figure~\ref{fig:nestedprism}.  Starting from an extensive self-stress, the framework's configuration can also be perturbed to try to generate a larger number of self-stresses -- particularly by exploiting symmetry, as discussed in this article -- although this process inherently disrupts the original extensive  self-stress.  This transition is where the notion of $\Gamma$-extensive self-stresses becomes relevant, as they tend to allow for a distributed and structurally beneficial self-stress space, whereas $\Gamma$-localised self-stresses typically remain undesirable.

\begin{remark}
    Extensive self-stresses are also relevant in the analysis of mechanical linkages that are pinned down to eliminate trivial motions. As shown in \cite{Assur}, a minimal pinned isostatic graph $G$ in the plane, also known as an \emph{Assur graph} \cite{Assur1}, can be characterised based on special realisations $(G,\p)$ that have a unique (up to scalar) self-stress, which is non-zero on all edges (and is hence extensive), as well as a unique (up to scalar) infinitesimal flex, which is is non-zero on all un-pinned vertices.
\end{remark}

One problem with the definition of extensive self-stresses is that finding them in sparse graphs is difficult. (In contrast, if $G$ is a rigidity circuit, i.e. a minimally dependent graph, then any generic realisation of it will have an extensive self-stress.)  
We suggest trying to find extensive self-stresses using a theory introduced by White and Whiteley in \cite{wwstresses} for isostatic graphs. 

\begin{definition}\label{def:WWpure}
Let $G$ be a $d$-isostatic graph.  Let $\x$ be a $d$-dimensional ``configuration'' of variables, and $e_{ij} = \x_j - \x_i$ the 
``edge vector'' of each edge.  
White and Whiteley describe a ``bracket polynomial'' $C_G(\x)$ in the edge vectors $e_{ij}$, 
called the \emph{pure condition} of $G$.  A bracket is a determinant of a 
$d\times d$ matrix which has, as its rows, a $d$-tuple of edge vectors.
The defining property of $C_G$ is that a framework $(G,\p)$ supports 
a self-stress if and only if $C_G(\p) = 0$.
\end{definition} 
Let us quickly describe White and Whiteley's work in context.  For any graph 
$G$, the set $S$ of self-stressed frameworks $(G,\p)$,
for any graph $G$, must be invariant under affine transformations of the
configuration space.  The first and second theorems of invariant theory (see, e.g., 
\cite{S08}) imply that the vanishing ideal of $S$ is generated by bracket polynomials.
When $G$ is, furthermore, isostatic, this ideal has a unique (up to scale) generator, 
which must be $C_G$, for which White and Whitely give a combinatorial formula.
The White and Whiteley formula is a ``$d$-fan sum'' over a ``tie down'' (the 
choice of tie-down is shown to change the fan sum by a factor that can be removed).

White and Whiteley's aim in their work was to generalise classical ``synthetic'' 
approaches to spotting and designing frameworks with self-stresses.  Here synthetic 
refers to statements that can be formulated in terms of coincidences between projective 
flats and points, their spans, and their intersections.  In the White and Whiteley 
setup, these statements correspond to $C_G$ having what is known as a 
factorisation in the Grassmann--Cayley algebra (for short a ``Cayley factorisation'').
Whether every pure condition admits such a factorisation is open, though the 
answer is strongly believed to be negative.  Sturmfels and Whiteley \cite{SW91}, 
describe a bracket polynomial that does not admit a Cayley factorization.  White 
and Whiteley \cite{wwstresses} show that the pure condition of the complete 
bipartite graph $K_{4,6}$, which is isostatic in dimension three, has an 
irreducible factor $Q$ corresponding to the statement that the ten vertices of 
$(G,\p)$ lie on a quadric surface.  This statement is widely believed not 
to admit a Cayley factorisation; see \cite{T24} for recent progress on this 
question.

Going further, White and Whiteley explored factoring $C_G$ in the bracket and polynomial rings.  
The hypersurface $V(C_G)$, consisting of configurations on which $C_G$ vanishes, is, in general, 
reducible, so we have that $C_G = f^{r_1}_1f^{r_2}_2\cdots f^{r_k}_k$ for some 
uniquely determined irreducible factors $f_i$, corresponding to the geometric irreducible 
components $V(f_i)$ of $V(C_G)$.  They show that these $f_i$ also have a representation as bracket 
polynomials, and, equivalently, that the $V(f_i)$  are affinely invariant.

Because these $V(f_i)$ are irreducible, they have generic behaviours. 
If $X$ is an irreducible algebraic set defined by rational polynomials, 
then $\x\in X$ is called \emph{generic} in $X$, if, 
whenever $P$ is a polynomial with rational coefficients and $P(\x) = 0$, then $P$ 
vanishes on $X$.  This is a way of saying that $\x$ is algebraically typical for $X$.
We need this definition because, unless $X$ is an affine space, 
$\x\in X$ means that $\x$ satisfies some 
rational polynomials, and being typical doesn't mean ``no polynomial'', 
just ``no extra polynomials''.
\begin{theorem}
Let $G$ be a $d$-isostatic graph.  The following quantities are the same for every $\p\in V(f_i)$ that is generic in 
$V(f_i)$ (which makes it not generic in the configuration space):
\begin{itemize}
    \item The dimension of the space of self-stresses of $(G,\p)$.
    \item The support of a  self-stress of $(G,\p)$.
\end{itemize}
\end{theorem}
One might guess that the self-stress dimension of any $V(f_i)$ is always one, but White and Whiteley
observed that in $3$-space, the pure condition of the complete bipartite graph $K_{4,6}$ has a factor 
corresponding to the smaller part of the bipartition becoming coplanar.  Such a framework has two 
linearly independent self-stresses \cite{BR,Wbipartite}.  Whether there is a $2$-dimensional  example like this is 
open.  White and Whiteley also show that if $r_i=1$, then $(G,\p)$
with $\p \in V(f_i)$ generic 
has exactly one self-stress  up to scale.  
(It is, however, unclear whether it is  extensive or not.) Closing the loop, in the $K_{4,6}$ example, 
the factor in question appears as a square.

\begin{remark}
    It is a conjecture of White and Whiteley \cite{wwstresses} that in dimension 2 the pure condition has a single factor $f$ (possibly raised to some power) if and only if the graph has no proper rigid subgraph on at least 3 vertices. Moreover, White and Whiteley conjecture that if there is a factor raised to some power, then that factor must have more than one self-stress.    
\end{remark}

There are a few directions to go from here:
\begin{itemize}
    \item By factoring the pure condition, we can consider, factor by factor, a typical (generic) 
        self-stressed framework $(G,\p)$ with $\p\in V(f_i)$.  If any such $(G,\p)$ has only one self-stress (up to scale) and it has full support, then, by definition, 
        every typical framework $(G,\q)$ with $\q\in V(f_i)$ will have the same 
        property.  Moreover, this self-stress will be extensive, so $V(f_i)$ yields a family 
        of extensively self-stressed frameworks.
    \item By factoring the pure condition, we can check whether each $V(f_i)$ is invariant under the 
        averaging map.  If not, typical self-stresses associated with $V(f_i)$ disappear after averaging.
        This gives us a new obstruction to using symmetrisation to increase the number of self-stresses.
        If $V(f_i)$ is invariant under the 
        averaging map, then symmetrisation starting from a framework associated with $V(f_i)$ may well 
        increase the number of independent self-stresses.
    \item There is a $(\Gamma,\tau)$-action on the polynomial and bracket rings, and so we can use 
        representation theory on it.  It is unclear whether we get anything beyond what we can 
        see with the block-diagonalisation of the rigidity matrix (see \cite{fow00,kangwai2000,sch10a,owen10}), but we also 
        don't know that there isn't. We will leave this for future investigation.
    \item We can consider an extension of the White and Whiteley pure condition to non-isostatic graphs, i.e. to a pure ideal of unexpectedly 
        self-stressed graphs, and  replicate the key elements of the theory.  Again, we will leave this for future work.  
        We do note that the over- and under-counted cases have a different character, and we would mainly be interested in the under-counted case.
\end{itemize}

In the next section, we will pursue the directions in the first two bullet points and incorporate them into heuristics for finding realisations with extensive self-stresses.

\section{Heuristics for creating self-stressed symmetric frameworks}\label{sec:alg}

\subsection{Finding symmetric realisations with large self-stress dimension}

To find (symmetric or non-symmetric) realisations of a graph $G$ in the plane with a large number of states of self-stress,  it is natural to first try to explore the self-stress dimension of  $(\Gamma,\tau)$-generic realisations of $G$, for all possible $\Gamma \subseteq \textrm{Aut}(G)$ and $\tau:\Gamma\to \text{Euc}_0(2)$. 

A key tool to explore this is the symmetry-extended Maxwell rule. This rule is originally due to Fowler and Guest \cite{fow00}, and has since been adapted and extended in various ways (see e.g. \cite{CFGSW09,smmb,mmsb,SGF14,foguow,SFG23}). It relies on the fact that the rigidity matrix of a $(\Gamma,\tau)$-symmetric framework can be block-decomposed for suitable symmetry-adapted bases in such a way that each block matrix corresponds to an irreducible representation $\rho_i$, $i=1,\ldots, t$, of the group \cite{fow00,kangwai2000,sch10a,owen10}. Moreover, the space of trivial infinitesimal motions decomposes into a direct sum of subspaces, each of which lies in the kernel of one of the block-matrices \cite{sch10a}. This leads to a refined Maxwell count -- one for each block matrix: while Maxwell's index theorem (recall Equation~(\ref{eq:maxwell})) arises from comparing the number of columns and rows in the entire rigidity matrix, we can now compare the number of columns and rows of each block-matrix, leading to additonal information.  

To apply the symmetry-extended Maxwell rule,  
one computes a number for each conjugacy class of the group, which depends on the number of vertices and edges that are fixed by elements in that conjugacy class \cite{fow00,CFGSW09}. (Here a vertex $v$ is \emph{fixed} by $\gamma\in\Gamma$ if $\gamma(v)=v$, and an edge $v_1v_2$ is fixed by $\gamma$ if $\gamma(v_i)=v_i$ for $i=1,2$, or $\gamma(v_1)=v_2$ and $\gamma(v_2)=v_1$.)
For a given ordering of the classes, this forms a vector (a ``character" in the language of group representation theory) which can be written uniquely as a linear combination $\sum_{i=1}^{t}\alpha_i\chi(\rho_i)$ of the characters $\chi(\rho_i)$ of the irreducible representations $\rho_1,\ldots, \rho_t$. The characters $\chi(\rho_i)$ can be read off from standard character tables -- see e.g. \cite{atk70,alt94}.  Remarkably, each $\alpha_i$ is an integer with the  property that if $\alpha_i>0$, then any $(\Gamma,\tau)$-symmetric framework must have  at least $\alpha_i\cdot \textrm{dim }{\rho_i}$ independent infinitesimal flexes (exhibiting the symmetry described by $\rho_i$), and if $\alpha_i<0$, then any $(\Gamma,\tau)$-symmetric framework must have  at least $|\alpha_i|\cdot \textrm{dim }{\rho_i}$ independent self-stresses of symmetry type $\rho_i$. Here $\textrm{dim }{\rho_i}$ denotes the dimension of the irreducible representation $\rho_i$.

So for any pair $(\tau,\Gamma)$ this provides us with  a lower bound on the self-stress dimension for a $(\tau,\Gamma)$-symmetric realisation of $G$. This suggests the following heuristic, which we will call Algorithm 1. 

Since we are mainly interested in engineering applications such as gridshell design, we focus on $2$-dimensional planar graphs. 
In Step (4) of Algorithm 1 we also avoid highly degenerate graph realisations arising from combinatorial symmetries that force vertices to coincide or edges to overlap. 

\medskip

\noindent \textbf{Algorithm 1}:
\begin{enumerate}
\item Input: planar graph $G$.

\item  Enumerate all possible subgroups $\Gamma< \text{Aut}(G)$.

\item For each $\Gamma$ found in Step (2), enumerate all $\tau:\Gamma\to \text{Euc}_0(2)$ that provide a geometric realisation of $\Gamma$ in $\mathbb{R}^2$.

\item For each pair $(\Gamma,\tau)$ identified in Step (3), discard this pair if there are two   vertices $u\neq v$ of $G$ that are fixed by $\gamma\in \Gamma$, where $\tau(\gamma)$ is a rotation, or if for some $\gamma\in \Gamma$, where $\tau(\gamma)$ is a reflection, the subgraph of $G$  induced by the vertices that are fixed by $\gamma$ is not  a disjoint union of paths.

\item For each pair $(\Gamma,\tau)$ that remains after Step (4), compute $\sum_{i=1}^{t}\alpha_i\chi(\rho_i)$. 

\item For each of the pairs $(\Gamma,\tau)$ from Step (5), compute the total detected self-stress dimension by adding up  $|\alpha_i|\cdot \textrm{dim }\rho_i$ over all $\alpha_i<0$. 

\item  Output: list of pairs $(\Gamma,\tau)$ 
ordered by the size of the corresponding self-stress dimensions detected in Step (6). (The self-stress dimensions of each symmetry type can also be given). 

\end{enumerate}

\medskip

Step (2) may be done using existing  software packages such as  NAUTY \cite{NAUTY}. For Step (3) we identify, for each group $\Gamma$ found in Step (2), the list of point groups in $\mathbb{R}^2$ it is isomorphic to. To do this, we may determine the generators of each $\Gamma$ and assign them to the generators of each possible point group in $\mathbb{R}^2$. This is feasible, as for $d=2$ (and for $d=3$ if one is interested in extending the heuristic to $3$-space),  there are well known classifications of all point groups, as described in \cite{atk70,alt94}, for example. Overall, Step (3) may be carried out using available computational algebra systems, such as Magma \cite{magma}, GAP \cite{GAP4}, 
Oscar \cite{OSCAR}, or Mathematica \cite{Mathematica}. 
For Step (4) we simply identify, directly from the graph automorphisms, the  vertices and edges of $G$ that are fixed by the relevant elements $\gamma$ of $\textrm{Aut}(G)$, and then check the corresponding isometries $\tau(\gamma)$.

Since in applications like gridshell design, one is often interested in graph realisations without edge crossings, one could  go further here and rule out pairs $(\Gamma,\tau)$ that obviously force edge crossings: if there exists $\gamma\in\Gamma$, where $\tau(\gamma)$ is a reflection, and there are two distinct edges $e,f$ of $G$ with $\gamma(e)=f$, then $(\Gamma,\tau)$ should be discarded, as the two edges must cross at a point on the reflection line. Similarly, if there are two distinct edges that are fixed by $\gamma$, where $\tau(\gamma)$ is the half-turn, or if there is a vertex  that is fixed by $\gamma'$, where $\tau(\gamma')$ is a non-trivial rotation, and an edge  that is fixed by $\gamma''$, where $\gamma''$ is the half-turn, then $(\Gamma,\tau)$ should also be discarded, as either two edges cross or an  edge goes through a vertex at the centre of rotation (the origin). 
 (In an extension of Algorithm 1 to $3$-space, one would simply check the analogous conditions to the ones in Step (4), which force degenerate realisations, or realisations with obvious edge crossings in $3$-space.) To further investigate whether a pair $(\Gamma,\tau)$ forces edge crossings, one could apply the corresponding symmetric averaging map to random  realisations of $G$ and check for planarity. Note that deciding whether a given framework has no edge crossings can be solved in time $O(n\log n)$ using standard methods from classical computational geometry \cite{CGbook,cgal:eb-24a}, but even a naive $O(n^2)$ algorithm would not be a bottleneck.

For Step (5) one only needs to identify, again directly from the graph automorphisms, the number of vertices and edges of $G$ that are fixed by the relevant elements of $\textrm{Aut}(G)$ and apply well-known formulas from group representation theory to find the coefficients $\alpha_i$ (see for example \cite[Theorem 4.4]{sch10a}). Steps (6) and (7) are then trivial. It is important to keep in mind that the maximum self-stress dimension found by this algorithm is only a lower bound for the maximum number of self-stresses with the corresponding symmetries, as there might be additional self-stresses that are not detectable with the symmetry-extended Maxwell rule.

Note that if, for a given graph $G$, we detect a self-stress  when we apply the symmetry-extended Maxwell rule with a point group $\tau(\Gamma)$, then we also detect the same self-stress  if we apply the rule with a larger group $\tau(\Gamma')>\tau(\Gamma)$. In fact, applying the rule with the larger group will reveal additional symmetry properties of the self-stress, and it may also find additional self-stresses. 
If we have any two point groups of different sizes, with neither group being contained in the other, then both groups may still reveal the same number of self-stresses, but  with different symmetry properties. Recall that 
 a group $\Gamma$ may also have different geometric realisations $\tau:\Gamma\to \text{Euc}_0(d)$, which can lead to different self-stress properties. (In dimension $2$, this only happens if $\Gamma$ has order 2, which geometrically is either a reflection or half-turn group. In 3-space, however, there are more examples like this.)

\begin{example} Consider for example the triangular prism graph. The pairs $(\Gamma,\tau)$ with \emph{planar} framework realisations are illustrated in Figure~\ref{fig:desarguestypes}. The half-turn and three-fold rotational group (with the given choices of $\Gamma$) do not induce any self-stress under symmetry-generic conditions (Figure~\ref{fig:desarguestypes}(c) and (e)). The reflection symmetry, and the dihedral symmetries of order 4 and 6 ((Figure~\ref{fig:desarguestypes}(a), (b), (d) and (f)) each induce a single self-stress, which can be detected in each case with the symmetry-extended Maxwell rule.
\end{example}

Once a desired symmetry has been identified with Algorithm 1,  the symmetric averaging map, applied to some realisation of $G$, can be used to construct a configuration with that symmetry and at least that number of self-stresses.  The symmetric realisation may then be analysed further to see if additional self-stresses can be detected, or to obtain extra information about the self-stresses (e.g. regarding $\Gamma$-extensiveness or even general extensiveness --  see Algorithm 4 below). Of course, one may also try to perturb the configuration (or even slightly alter the graph) to try to create additional self-stresses (see e.g. \cite{mmsb,mazbak}).

In practice, a designer may provide an initial configuration for a given graph which is to be optimised locally. In such cases one could simply explore the symmetries of nearby configurations (which can easily be identified by inspection) rather than all possible symmetries. The symmetric averaging map can again be used to produce the nearby symmetric frameworks.

\subsection{Finding  realisations with extensive self-stresses} 

 In this section, we rely on the theory described in Section~\ref{subsec:ssext} to establish an algorithm that checks  for general extensive self-stresses in (not necessarily symmetric) realisations of a given graph. In the following, we consider an arbitrary dimension $d$, but only $d$-isostatic graphs.

\medskip

\noindent \textbf{Algorithm 2}:

\begin{enumerate}
\item Input: $d$-isostatic graph $G$.
    \item Compute the pure condition $C_G$ for $G$ and factor it.
    \item For each factor $f_i$ obtained in Step (2), compute the self-stress dimension and support of self-stresses for $\p\in V(f_i)$ generic.
    \item Output: The factors $f_i$ with an extensive self-stress and otherwise report failure.
\end{enumerate}
    
\medskip

\begin{example} When we apply Algorithm 2 to the triangular prism graph (for $d=2$) again, then we find in Step (2) that the pure condition has three factors \cite{wwstresses}. Geometrically interpreted, two of the factors correspond to the points of the two triangles being collinear, and the third factor says that the bars connecting the two triangles must meet in a point. The generic self-stress dimension for each factor is $1$. While the first two factors have a self-stress that is only supported on a triangle, the third factor has generically full support and hence corresponds to an extensive self-stress. 
\end{example}

\medskip

We now mention an alternative to Algorithm 2, which is based on the well-known ``force density'' \cite{schek} or ``rubber-band'' 
\cite{lovasz,LLW,tutte} method.  The following algorithm is not a novelty to this paper, but we recall it 
just to set up the context.

\medskip

\noindent \textbf{Algorithm 3 (rubber-banding)}:

\begin{enumerate}
\item Input: $(d+1)$-connected graph $G = (V,E)$.
    \item Partition $V$ into ``boundary'' and ``interior'' vertices 
        $V^\partial\cup V^\circ$, where $|V^\partial| = d+1$.
    \item Assign points $\p^\partial = (\p_i)_{i\in V^\partial}$ in $\mathbb{R}^d$ to the boundary vertices so that 
        they are in general position.
    \item Assign edge weights $\mathbf \omega^\circ = (\mathbf \omega_{ij})$ to the edges with at least one 
        endpoint in $V^\circ$.
    \item Solve for the positions of the interior vertices, using the inhomogeneous linear system
    which has an equation
    \[
        \sum_{j :ij\in E}  \omega_{ij}(\p_i - \p_j) = 0
    \]
    for each vertex $i\in V^\circ$ and variables $\p_i$ for each $i\in V^\circ$.  
    If there is a solution, we get a configuration
    $\p = (\p^\partial, \p^\circ)$ in equilibrium at the interior vertices.  If there is not, 
    report failure and stop.
    \item Try to solve for the coefficients of a self-stress on the edges induced by the boundary 
        using the inhomogeneous linear system which has an equation 
    \[
        \sum_{j :ij\in E}  \omega_{ij}(\p_i - \p_j) = 0
    \]
    for each vertex $j\in V^\partial$ and the variables are the edge weights $\omega_{ij}$, 
    for edges induced by $V^\partial$; i.e., both $i,j\in V^\partial$.
    \item Output: $(\mathbf \omega,\p)$ if a solution exists and otherwise report failure.
\end{enumerate}

Algorithm 3 has received a lot of study in many communities (it is referred to as `rubber banding' by those in the rigidity theory community and the `force density method' by those in the engineering community).  Let us first comment 
on how the algorithm can fail.  The system in Step (5) is inhomogeneous, since some of 
the $\p_j$ appearing on the lhs correspond to vertices $j\in V^\partial$.  However, 
if $G$ is $(d+1)$-connected and the weights $\omega_{ij}$ in $\omega^\circ$ are positive, 
then the system is invertible  \cite{tutte,schek}, so it will have a unique solution (the number of 
variables and equations are both $d|V^\circ|$).  When 
the entries of $\omega^\circ$ have both signs, then it is possible for the system to 
be singular, in which case, whether there is a solution depends on the positions 
of the vertices in $\p^\partial$, since these determine the rhs. However, for generically 
chosen $\omega^\circ$, 
the system will be invertible \cite{schek,LLW}.  

The system in Step (6) is also inhomogeneous,
so there is a possibility that it is inconsistent.  The analysis of this step 
relies on the structure of the subgraph $G^\partial$ induced by $V^\partial$ and also a 
statics argument.  The key point is that we need the framework $(G^\partial,\p^\partial)$
to be able to resolve the load placed on it by the edges connecting to interior 
vertices.  When $\omega^\circ$ is either positive or generic and $G^\partial$ is 
isomorphic to $K_{d+1}$ it is a folklore fact that this happens (see, e.g., \cite{gstr})
when $\p^\partial$ is in general position.  The general case, especially when $G^\partial$
is not generically rigid, is a complicated non-linear problem that depends on both 
$\p^\partial$ and $\omega^\circ$.
The condition of $G^\partial$ inducing a $K_{d+1}$ can always be made to hold 
for either a $3$-connected planar graph in dimension $2$ or its planar dual graph, 
since one of them must contain a triangle (i.e., a $K_{3}$).  Hence, we can apply 
Algorithm 3 either directly or to the dual graph, and then recover a self-stressed 
framework using the Maxwell reciprocal \cite{maxwell1870reciprocal,cremona1872figure}.

It is shown in \cite{gstr} that, when the boundary vertices induce a $K_{d+1}$ subgraph, 
applying 
Algorithm 3 and varying $\omega^\circ$ produces an irreducible set of self-stresses for $G$.  The generic 
self-stresses in this set have maximum rank stress matrices \cite{connelly82}.  
Moreover, a generic framework satisfying a generic self-stress 
from Algorithm 3\footnote{This notion has meaning from irreducibility.} will be in general 
position, and it will have a space of self-stresses with a constant dimension called 
the \emph{stressed corank} of $G$.  This, in particular, implies that Algorithm 3 
produces extensive self-stresses for either almost all or almost no choices of $\mathbf \omega^\circ$,
because the output self-stress will be extensive exactly when it is the only self-stress, up to 
scale, for the output framework.  

Algorithm 2, while more difficult to implement, has a more refined behaviour.  There can be factors 
of the pure condition associated with self-stresses that do not have maximum rank stress 
matrices, which Algorithm 3 does not find.  Moreover, even if the stressed corank of $G$ 
is greater than one, in which case Algorithm 3 does not produce extensive self-stresses, 
Algorithm 2 may still be able to find them.  In general, the relationship between 
Algorithm 2 and Algorithm 3 is unclear.  For example, we do not know that the self-stresses 
from Algorithm 3 all correspond to a single factor of the pure condition.
It is left to future work to develop an explicit set of conditions for the guaranteed success of Algorithm 3.

\subsection{Finding symmetric realisations with extensive self-stresses} 

We now combine the heuristics from the previous two subsections  to  find \emph{symmetric} frameworks with extensive self-stresses. We focus on $2$-isostatic graphs.

\medskip

\noindent \textbf{Algorithm 4}:

\begin{enumerate}
    \item Input: $2$-isostatic graph $G$.
    \item Carry out Algorithm~1 to identify possible symmetry pairs $(\Gamma,\tau)$ with non-degenerate $(\Gamma,\tau)$-symmetric  framework realisations.
\item Carry out Algorithm~2 to determine the factors $f_i$ of the pure condition of $G$ with an extensive self-stress.
\item For each factor $f_i$ found in Step (3), and for each symmetry pair from Step (2), check, for some  $\p\in V(f_i)$, if $A\p \in V(f_i)$ for the corresponding symmetric averaging map $A$. This identifies symmetries  that induce an extensive self-stress.
\item Output: symmetries $(\Gamma, \tau)$ providing an extensive self-stress and otherwise report failure.
    \end{enumerate}

\begin{example} We again consider the triangular prism graph. If we apply Algorithm 1 to this graph, we find that the reflectional symmetry $\mathcal{C}_s$ indicated in Figure~\ref{fig:prism}(b), for instance, yields a (fully-symmetric) self-stress. Moreover, Algorithm 2 concludes that if the graph is realised generically with respect to one of the factors $f_i$ of the pure condition, then it has an extensive self-stress.  Applying Algorithm 4 to this graph, we find that $V(f_i)$ is  invariant under the averaging map for $\mathcal{C}_s$, and hence we conclude that $\mathcal{C}_s$-symmetric realisations of the prism graph typically have a (fully-symmetric) extensive self-stress. 
\end{example} 

The interaction between representation-theoretic symmetry methods embodied by Algorithm 1 and the algebraic 
ones in Algorithm 2 can be quite subtle.  The  example in Figure~\ref{fig:pure} has the property that 
a \emph{generic} framework of its symmetry type will be self-stress free.  However, there is (a necessarily non-generic) 
framework with this symmetry type that has an extensive self-stress.  Algorithm 2 will predict that there is an
extensively self-stressed framework but, since the symmetric frameworks are non-generic even in the 
relevant ``$V(f_i)$'', 
Algorithm 4 will fail to find the symmetric, extensively self-stressed framework.
In order to overcome this limitation, we will need to develop a specialisation of White and Whiteley's
theory that applies to symmetric configuations.

To find special self-stressed configurations within sets of  frameworks with specified symmetries (like the one in Figure~\ref{fig:pure}), it is natural to look for a symmetry-adapted version of the pure condition  for the phase-symmetric orbit rigidity matrices  corresponding to the  block matrices of the block-diagonalised rigidity matrix \cite{sw2011,bt2015}. In the non-symmetric setting, the pure conditions were extracted by using ``tie-downs"  to square up the rigidity matrix. While we expect a similar approach to work for the orbit rigidity matrices, it will require a detailed analysis to determine, for each pair $(\Gamma,\tau)$, the tie-downs that eliminate all  trivial infinitesimal motions of each symmetry type corresponding to an irreducible representation of the group.

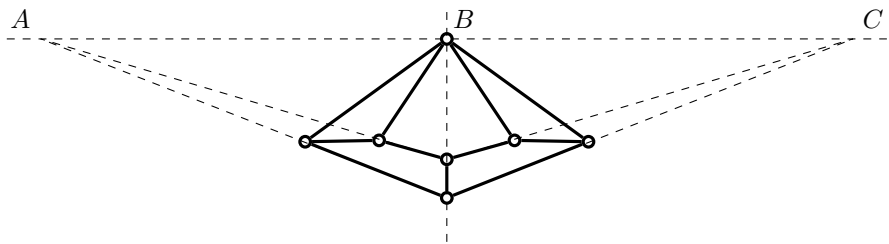
\begin{figure}[htp]
\begin{center}
\begin{tikzpicture}[very thick,scale=0.9]
\tikzstyle{every node}=[circle, draw=black, fill=white, inner sep=0pt, minimum width=4pt];
   
       \path (90:2.2cm) node (p1)  {} ;
     \path (162:2.2cm) node (p2)  {} ;
     
      \path (18:2.2cm) node (p5)  {} ;
      
     \path (0,-0.15) node (p6)  {} ;
     \path (0,0.42) node (p7)  {} ;
      \path (-1,0.7) node (p8)  {} ;
     \path (1,0.7) node (p9)  {} ;

       \node [draw=white, fill=white] (b) at (0.25,2.5) {$B$};
       \node [draw=white, fill=white] (b) at (-6.3,2.5) {$A$};
       \node [draw=white, fill=white] (b) at (6.3,2.5) {$C$};

     \draw[dashed,thin] (-6.5,2.2)--(6.5,2.2);
      \draw[dashed,thin] (p7)--(6,2.2);
      \draw[dashed,thin] (p6)--(6,2.2);
        \draw[dashed,thin] (p7)--(-6,2.2);
      \draw[dashed,thin] (p6)--(-6,2.2);

     \draw(p1)--(p2);
   \draw(p1)--(p8);
   \draw(p1)--(p9);
   \draw(p1)--(p5);
     \draw(p8)--(p2);
   \draw(p5)--(p9);
   \draw(p6)--(p7);
   \draw(p6)--(p2);
    \draw(p6)--(p5);
       \draw(p8)--(p7);
    \draw(p7)--(p9);

    \draw[dashed, thin](0,-0.8)--(0,2.7);

        \end{tikzpicture}
      \end{center}
\vspace{-0.6cm} \caption{A generically $2$-isostatic graph which is also $\mathcal{C}_s$-generically isostatic. So Algorithm 1 does not detect any self-stress. Algorithm 2 finds a a factor corresponding to an extensive self-stress. As shown in the figure, the graph can be realised with $\mathcal{C}_s$ symmetry so that it satisfies the pure condition (points $A$, $B$ and $C$ are collinear), but this is not the case for a $\mathcal{C}_s$-generic realisation. Hence Algorithm 4 will not lead us to the realisation in the figure.
} \label{fig:pure}
\end{figure}

\section{Modelling notes}\label{sec:model}

The intended application of the methods and results of this paper is to a problem in efficient roof design.  The 
modeling step has some subtleties that are well-understood by engineers,
but less so by rigidity theorists, so we briefly note them.

The basic setup is that we have a framework $(G,\tilde{\p})$ in $\RR^3$ that 
is the $1$-skeleton of a polyhedral surface.  Moreover, we assume that 
$(G,\tilde{\p})$ is, in fact, a Maxwell--Cremona lifting of some 
planar drawing $(G,\p)$ of its $1$-skeleton.
In particular, we have that $(G,\p)$ must have, at least, the 
self-stress that gives rise to the flat faces of its 
lifting to $(G,\tilde{\p})$.

Let us now consider applying a ``downward'' load $\mathbf f$ to $(G,\tilde{\p})$.
This means that each $\mathbf f_i$ is of the form 
\[
   \mathbf f_i = \beta_i \e_3 \qquad \text{(some $\beta_i\in \RR$)}
\]
with $\e_3=(0,0,1)^\top$ and not all the $\beta_i = 0$.  If $(G,\tilde{\p})$ resolves $\mathbf f$, 
we have an equation of the form 
\[
    \mathbf \omega^\top R(\tilde{\p}) = \mathbf f^\top 
\]
for some $\mathbf \omega\in \RR^m$.  The form of $\mathbf f$ now implies that 
\[
    \mathbf \omega^\top R(\p) = 0
\]
In other words: resolutions of downward (or upward) forces by $(G,\tilde{\p})$
give rise to self-stresses of the projected framework $(G,\p)$.  

This correspondence works in reverse as well, but there is a subtlety.  Given 
some arbitrary self-stress $\mathbf \omega$ of $(G,\p)$, it is 
true that 
\[
    \mathbf \omega^\top R(\tilde{\p}) =: \mathbf g  
\]
has the form 
\[
   \mathbf g_i = \beta_i \e_3\qquad \text{(some $\beta_i \in \RR$).}
\]
However, the construction reveals that we do not have very much 
control over $\mathbf g$.  In the worst case, $\mathbf \omega$ is also a self-stress of $(G,\tilde{\p})$ itself, and $\mathbf g$ is even the zero load.

More generally, the modeling problem of resolving application-relevant 
loads is only partially captured by the question of maximising the 
number of self-stresses in $(G,\p)$.  To describe the 
general version, let us define $\mathcal{D}$ to be the 
subspace of $\mathbf f\in \RR^{3n}$ so that $\mathbf f_i = \beta_i \e_3$ for all $i$,
and $\mathcal{T}$ to be the space of trivial infinitesimal motions at $\tilde{\p}$
in $3$-space.  We then have a space
\[
    F  \subseteq \mathcal{D}/\mathcal{T}
\]
of target loads to resolve. Then we need that 
\[
    (\ker R(\p)^\top)R(\tilde{\p})/\mathcal{T} \supseteq F/\mathcal{T}
\]
In other words, we need that for each $\mathbf f\in F$, there is a 
self-stress $\mathbf \omega$ of $(G,\p)$ which can hit 
$\mathbf f$ when we put it to the left of the rigidity matrix of 
$(G,\tilde{\p})$.

To expand on why this is interesting, we are \emph{not}
requiring that $(G,\tilde{\p})$ is infinitesimally rigid as a 
$3$-dimensional framework.  If that were the case, as happens
with a triangulated surface, we would be able to hit 
\emph{any} load orthogonal to the trivial motions at $\tilde{\p}$ 
with some $\mathbf \omega$.  Looking at the zero pattern, this implies that the 
pre-image of $F$ under this map contains the self-stresses 
of $(G,\p)$.  Our situation is more interesting: since $(G,\tilde{\p})$ 
doesn't resolve every load, whether or not it does resolve the 
ones in $F$ depends on the geometry of $(G,\tilde{\p})$ and 
not just the dimension of the space of self-stresses 
of $(G,\p)$.

\section{Practical considerations}\label{sec:practice}

When we fabricate a framework, there will be some error relative to 
an idealised design.  In our setting, this poses problems, again, as 
in the previous section, our frameworks $(G,\tilde{\p})$ are 
not infinitesimally rigid.  Even if $(G,\tilde{\p})$ is 
generic (i.e., its rigidity matrix has maximum rank over 
all $3$-dimensional configurations), it may be that a 
small perturbation, while also generic, cannot resolve a 
specific load $\mathbf f$.  The guarantee of genericity is only about 
the dimension of the space of resolvable loads, and, in the 
case where this is strictly less than the dimension of all 
equilibrium loads, changing $\tilde{\p}$ can change which 
loads are resolvable.  (Things are even more complicated when 
$(G,\tilde{\p})$ is not generic.)

From the discussion above, it is too much to assume that a 
less-than-ideal realisation of the design perfectly 
resolves the target load $\mathbf f$.  Instead, we present here 
a simple linear analysis of how error in the configuration 
propagates to resolving  $\mathbf f$.  The approach is 
very simple: we use the same $\mathbf \omega$ from the ideal design 
and see how close a perturbed framework comes to resolving the 
target load \emph{using this $\mathbf \omega$}.  (We don't try to 
find the best $\mathbf \omega$ for the perturbed structure and the 
target load $\mathbf f$, which requires a non-linear analysis.)

Suppose that we have some $\mathbf \omega\in \RR^m$ so that 
\[
    \mathbf \omega^\top R(\tilde{\p}) = \mathbf f^\top
\]
and that we have some tolerance $\eps$ in the sense that the structure 
can resolve loads of norm $\eps$ in some other way (e.g., stiffness in the
joints through moment-resistant connections or other elastic considerations).  Let us consider replacing $\tilde{\p}$ 
by $\tilde{\p} + \e$.  We can compute 
\[
     \mathbf \omega^\top R(\tilde{\p} + \e) = \mathbf f^\top +  \mathbf \omega^\top R(\e). 
\]
The norm of the error vector is at most 
\[
    \|R(\e)\|\|\mathbf \omega\|
\]
and we can estimate the norm of a rigidity matrix explicitly by 
picking a unit vector $\mathbf u$ in $\RR^{dn}$ and computing 
\[
    \mathbf u^\top R(\e)^\top R(\e)\mathbf u = \sum_{ij\in E(G)} \iprod{\e_i - \e_j}{\mathbf u_i - \mathbf u_j}^2
    \le\sum_{ij\in E(G)} \|\e_i - \e_j\|^2\|\mathbf u_i - \mathbf u_j\|^2 
    \le 4m \operatorname{diam}(\e)^2.
\]
It now follows that 
\[
    \|R(\e)\| \le 2\sqrt{m}\operatorname{diam}(\e)
\]
where the diameter of $\e$ is 
\[
    \max_{i\neq j} \|\e_i - \e_j\|.
\]
So if we pick $\e$ so that
\[
    \operatorname{diam}(\e) < \frac{\eps}{2\sqrt{m}\|\mathbf \omega\|}
\]
then the error is bounded by $\eps$.  
This seems plausible for use in practice, 
but a more refined linear analysis might improve the numbers.
It is interesting that the dependence on $\tilde{\p}$ is expressed 
through $\mathbf \omega$ and that we get a bound in terms of how 
much edge lengths changed, which seems more natural than 
absolute changes in the positions of the vertices.

\section*{Acknowledgements} We thank William Baker, Arek Mazurek and Jessica Sidman for helpful discussions. B.S. was partially supported by the ICMS Knowledge Exchange Catalyst Programme.

\bibliographystyle{abbrvnat}
\bibliography{avg}

\end{document}